\documentclass[]{interact}

\usepackage{amsmath,amsthm}
\usepackage{mathtools}
\usepackage{enumerate}
\usepackage{graphicx}
\usepackage[caption=false]{subfig}
\usepackage[english]{babel}
\usepackage[T1]{fontenc}
\usepackage{lmodern}
\usepackage[numbers,sort&compress]{natbib}
	\bibpunct[, ]{[}{]}{,}{n}{,}{,}
\renewcommand\bibfont{\fontsize{10}{12}\selectfont}
\makeatletter
	\def\NAT@def@citea{\def\@citea{\NAT@separator}}
\makeatother

\theoremstyle{plain}
\newtheorem{theorem}{Theorem}[section]
\newtheorem{lemma}[theorem]{Lemma}
\newtheorem{corollary}[theorem]{Corollary}
\newtheorem{proposition}[theorem]{Proposition}
\theoremstyle{definition}

\newtheorem{example}[theorem]{Example}
\theoremstyle{remark}
\newtheorem{remark}[theorem]{Remark}

\newcommand{\iprod}[2]{\langle#1,#2\rangle}
\newcommand{\iprodX}[2]{\langle#1,#2\rangle_{X}}
\newcommand{\iprodY}[2]{\langle#1,#2\rangle_{Y}}
\newcommand{\iprodZ}[2]{\langle#1,#2\rangle_{Z}}

	\let\norm=\enVert
\newcommand{\normX}[1]{\lVert#1\rVert_{X}}
\newcommand{\normY}[1]{\lVert#1\rVert_{Y}}
\newcommand{\normZ}[1]{\lVert#1\rVert_{Z}}

\newcommand{\rl}{\mathbb R}

\newcommand{\CC}{\mathcal{C}}
\newcommand{\LL}{\mathcal{L}}
\newcommand{\OO}{\mathcal{O}}

\DeclareMathOperator{\sri}{sri}
\DeclareMathOperator{\proj}{proj}

\DeclareMathOperator{\id}{Id}

\begin{document}
\title{%
	On the Arrow--Hurwicz differential system for linearly constrained convex
	minimization
}
\author{%
	\name{%
		Simon K.\ Niederl{\"a}nder\thanks{
			CONTACT Simon K.\ Niederl{\"a}nder. Email:
			niederlaender@ist.uni-stuttgart.de}}
	\affil{%
		Institute for Systems Theory and Automatic Control,
		University of Stuttgart,\\ Stuttgart, Germany}}
\maketitle

\begin{abstract}
	In a real Hilbert space setting, we reconsider the classical Arrow--Hurwicz differential system in view of solving linearly constrained convex minimization problems. We investigate the asymptotic properties of the differential system and provide conditions for which its solutions converge towards a saddle point of the Lagrangian associated with the convex minimization problem. Our convergence analysis mainly relies on a `Lagrangian identity' which naturally extends on the well-known descent property of the classical continuous steepest descent method. In addition, we present asymptotic estimates on the decay of the solutions and the primal-dual gap function measured in terms of the Lagrangian. These estimates are further refined to the ones of the classical damped harmonic oscillator provided that second-order information on the objective function of the convex minimization problem is available. Finally, we show that our results directly translate to the case of solving structured convex minimization problems. Numerical experiments further illustrate our theoretical findings.
\end{abstract}


\begin{amscode}
	37N40; 46N10; 49M30; 65K05; 90C25
\end{amscode}

\section{Introduction}\label{sec1}
Let $X$ and $Y$ be real Hilbert spaces endowed with inner products $\iprodX{\,\cdot\,}{\,\cdot\,}$, $\iprodY{\,\cdot\,}{\,\cdot\,}$ and induced norms $\normX{\,\cdot\,}$, $\normY{\,\cdot\,}$. Consider the minimization problem
\begin{equation}
\renewcommand{\theequation}{P}\tag{\theequation}\label{sec1:cvxproblem}
	\inf\,\{f(x)\mid Ax-b=0_{Y}\},
\end{equation}
where $f:X\to\rl$ is a convex and continuously differentiable function, $A:X\to Y$ a linear and continuous operator, and $b\in Y$. We associate with \eqref{sec1:cvxproblem} the Lagrangian
\begin{align*}
	L:X\times Y&\longrightarrow\rl\\
	(x,\lambda)&\longmapsto f(x)+\iprodY{\lambda}{Ax-b}
\end{align*}
which, by construction, is a convex-concave and continuously differentiable bifunction. A pair $(\bar{x},\bar{\lambda})\in X\times Y$ is a saddle point of the Lagrangian $L$ if
\begin{equation*}
	L(\bar{x},\lambda)\leq L(\bar{x},\bar{\lambda})\leq L(x,\bar{\lambda}),
	\quad\forall(x,\lambda)\in X\times Y.
\end{equation*}
It is well known that $(\bar{x},\bar{\lambda})\in X\times Y$ is a saddle point of $L$ if and only if $\bar{x}$ is a minimizer of \eqref{sec1:cvxproblem}, $\bar{\lambda}$ is a maximizer of the Lagrange dual to \eqref{sec1:cvxproblem}, that is
\begin{equation}
\renewcommand{\theequation}{D}\tag{\theequation}\label{sec1:dualcvxproblem}
	\sup\,\{-f^{\ast}(-A^{\ast}\lambda)-\iprodY{\lambda}{b}\mid\lambda\in Y\},
\end{equation}
and the optimal values of \eqref{sec1:cvxproblem} and \eqref{sec1:dualcvxproblem} coincide; see, e.g., Ekeland and T{\'e}mam \cite{IE-RT:99}. Here, $f^{\ast}:X\to\rl\cup\{+\infty\}$ denotes the Fenchel conjugate of $f$ defined by $f^{\ast}(u)=\sup\,\{\iprodX{u}{x}-f(x)\mid x\in X\}$, and $A^{\ast}:Y\to X$ refers to the adjoint operator of $A$. Equivalently, $(\bar{x},\bar{\lambda})\in X\times Y$ is a saddle point of $L$ if and only if $(\bar{x},\bar{\lambda})$ solves the system of primal-dual optimality conditions
\begin{equation*}
	\begin{cases}
		\nabla f(x)+A^{\ast}\lambda=0_{X}\\
		\hspace{30pt}Ax-b=0_{Y}
	\end{cases}
\end{equation*}
with $\nabla f$ denoting the gradient of $f$. Throughout the text, we denote by $S\times M\subset X\times Y$ the (possibly empty) set of saddle points of $L$. We recall that a saddle point of $L$ exists whenever \eqref{sec1:cvxproblem} admits a minimizer and, for instance, the constraint qualification
\begin{equation*}
	b\in\sri A(X)
\end{equation*}
is verified\footnote{%
	We remark that, in the finite-dimensional case, the condition amounts to $b\in A(X)$ which is commonly re- ferred to as Slater assumption; see, e.g., Hiriart-Urruty and Lemar{\'e}chal \cite{JBHU-CL:93}.}.
Here, for a convex set $C\subset Y$, we denote by
\begin{equation*}
	\sri C=\{x\in C\mid\bigcup_{\mu>0}\mu(C-x)~
	\text{is a closed linear subspace of $Y$}\}
\end{equation*}
its strong relative interior; see, e.g., Bauschke and Combettes \cite{HHB-PLC:17}. We further recall that \eqref{sec1:cvxproblem} admits a minimizer whenever its feasible set is non-empty and, for instance, $f$ is coercive, that is, $\lim_{\normX{x}\to+\infty} f(x)=+\infty$. On the other hand, if the feasible set of \eqref{sec1:cvxproblem} is non-empty and $f$ is strongly convex, then \eqref{sec1:cvxproblem} admits a unique minimizer.

In this work, we reconsider the classical Arrow--Hurwicz differential system
\begin{equation}
\renewcommand{\theequation}{AH}\tag{\theequation}\label{sec1:arrowhurwicz}
	\begin{cases}
		\dot{x}+\nabla f(x)+A^{\ast}\lambda=0_{X}\\
		\hspace{30pt}\dot{\lambda}+b-Ax=0_{Y}
	\end{cases}
\end{equation}
relative to the convex minimization problem \eqref{sec1:cvxproblem}. The \eqref{sec1:arrowhurwicz} differential system was in essence originated by Arrow and Hurwicz \cite{KJA-LH:51} (see also Kose \cite{TK:56}, Arrow et al. \cite{KJA-LH-HU:58}) and is known to be intimately related to the mini-maximization of the Lagrangian $L$ associated with \eqref{sec1:cvxproblem}. Indeed, given the above system of primal-dual optimality conditions, we immediately observe that the zeros of the operator
\begin{align*}
	T:X\times Y&\longrightarrow X\times Y\\
	(x,\lambda)&\longmapsto(\nabla_{x}L(x,\lambda),-\nabla_{\lambda}L(x,\lambda)),
\end{align*}
that is, the `generator' of the \eqref{sec1:arrowhurwicz} differential system, are precisely the saddle points of the Lagrangian $L$, i.e.,
\begin{equation*}
	(\bar{x},\bar{\lambda})\in S\times M\quad\iff\quad
	T(\bar{x},\bar{\lambda})=(0_{X},0_{Y}).
\end{equation*}
Moreover, the operator $T$ is maximally monotone on $X\times Y$ as it is both monotone and continuous; cf. Minty \cite{GJM:62}. Therefore, $S\times M$ can be interpreted as the set of zeros of the maximally monotone operator $T$ and, as such, it is a closed and convex subset of $X\times Y$. The latter may also be deduced more elementary from the convexity-concavity properties of the `saddle function' $L$; cf. Rockafellar \cite{RTR:69a}.

\subsection{Preliminary facts}
As emphasized by Rockafellar \cite{RTR:71}, the general theory for semi-groups of contractions generated by maximally monotone operators (see, e.g., Crandall and Pazy \cite{MGC-AP:69}, Br{\'e}zis \cite{HB:73}) applies to the Arrow--Hurwicz differential system \eqref{sec1:arrowhurwicz}. These results, dating back to the works of Kato \cite{TK:67} and K{\=o}mura \cite{YK:67} (see also Browder \cite{FEB:76}), imply that the Cauchy problem associated with \eqref{sec1:arrowhurwicz} is well posed and that its (classical) solutions $(x,\lambda),(y,\eta):[0,+\infty)\to X\times Y$ verify the `non-expansiveness property'
\begin{equation*}
	\frac{\mathrm{d}}{\mathrm{d}t}\big(\normX{x(t)-y(t)}^{2}+\normY{\lambda(t)-\eta(t)}^{2}\big)\leq0,\quad\forall t\geq0.
\end{equation*}
If, in addition, the set $S\times M$ is non-empty, then the solutions $(x(t),\lambda(t))$ of \eqref{sec1:arrowhurwicz} remain bounded and, in fact, weakly converge in average, as $t\to+\infty$, towards a saddle point of the Lagrangian $L$ (see Baillon and Br{\'e}zis \cite{JBB-HB:76}), i.e., there exists $(\bar{x},\bar{\lambda})\in S\times M$ such that
\begin{equation*}
	\frac{1}{t}\int_{0}^{t}(x(\tau),\lambda(\tau))\,\mathrm{d}\tau\rightharpoonup
	(\bar{x},\bar{\lambda})\ \text{as $t\to+\infty$}.
\end{equation*}

The (asymptotic) stability properties of the solutions of \eqref{sec1:arrowhurwicz} (in the sense of Lyapunov) were further investigated by Venets \cite{VIV:85} (see also Fl\aa{}m and Ben-Israel \cite{SDF-ABI:89}). These results suggest that the solutions of \eqref{sec1:arrowhurwicz} tend towards a saddle point of $L$ giv- en that, for any $(x,\lambda)\in X\times Y$ with $x\notin S$, it holds that
\begin{equation*}
	L(\bar{x},\lambda)\leq L(\bar{x},\bar{\lambda})<L(x,\bar{\lambda}),
	\quad\forall(\bar{x},\bar{\lambda})\in S\times M.
\end{equation*}
The above condition is, of course, trivially satisfied whenever $f$ is strictly convex. In the respective works, the authors further noted that the solutions $(x,\lambda):[0,+\infty)\to X\times Y$ of \eqref{sec1:arrowhurwicz} obey the `Lagrangian identity'
\begin{equation}\label{sec1:lagidentity}
	\frac{\mathrm{d}}{\mathrm{d}t}L(x(t),\lambda(t))+\normX{\dot{x}(t)}^{2}=
	\normY{\dot{\lambda}(t)}^{2},\quad\forall t\geq0.
\end{equation}
The identity, however, was not pursued any further due to its indefinite character. We remark that, in the unconstrained case of \eqref{sec1:cvxproblem}, the above identity reduces to the well-known `descent property'
\begin{equation*}
	\frac{\mathrm{d}}{\mathrm{d}t}f(x(t))+\normX{\dot{x}(t)}^{2}=0,
	\quad\forall t\geq0
\end{equation*}
associated with the classical continuous steepest descent method; see, e.g., Br{\'e}zis \cite{HB:73}, Aubin and Cellina \cite{JPA-AC:84}.

Finally, the exponential decay properties of the solutions of \eqref{sec1:arrowhurwicz} were investigated by Polyak \cite{BTP:70}. Using spectral arguments, the work provides conditions for which the solutions $(x(t),\lambda(t))$ of \eqref{sec1:arrowhurwicz} converge at an exponential rate, as $t\to+\infty$, towards a saddle point $(\bar{x},\bar{\lambda})$ of $L$, i.e., for which there exists $\rho>0$ such that
\begin{equation*}
	\normX{x(t)-\bar{x}}^{2}+\normY{\lambda(t)-\bar{\lambda}}^{2}=
	\OO\big(\mathrm{e}^{-\rho t}\big)\ \text{as $t\to+\infty$}.
\end{equation*}
The decay rate estimates are, however, not derived in an explicit form.

In this work, our objective is to recover, unify and extend some of the previous results on the classical Arrow--Hurwicz differential system \eqref{sec1:arrowhurwicz} in view of solving the linearly constrained convex minimization problem \eqref{sec1:cvxproblem}. Using tools from monotone operator theory, we focus our attention on the convergence properties of the solutions of \eqref{sec1:arrowhurwicz} and further aim to characterize their limit within the set of saddle points of the Lagrangian. We also intend to make a contribution to the issue of finding (explicit) decay rate esti- mates on the solutions of \eqref{sec1:arrowhurwicz}.

\subsection{Presentation of the results}
The mini-maximizing properties of the solutions $(x,\lambda):[0,+\infty)\to X\times Y$ of \eqref{sec1:arrowhurwicz} with respect to the convex minimization problem \eqref{sec1:cvxproblem} and its associated Lagrange dual \eqref{sec1:dualcvxproblem} are conveniently measured in terms of the `primal-dual gap function'
\begin{equation*}
	t\longmapsto L(x(t),\,\cdot\,)-L(\,\cdot\,,\lambda(t))
\end{equation*}
relative to the set $S\times M$. Whenever the function $f$ is convex, we observe that the solutions $(x(t),\lambda(t))$ of \eqref{sec1:arrowhurwicz} may fail to converge as $t\to+\infty$ even though the set of saddle points of $L$ is comprised of a single element. As a consequence, it is natural to first study the average behavior of a solution of \eqref{sec1:arrowhurwicz}. Using the notion of the Ces{\`a}ro average $(\sigma,\omega):(0,+\infty)\to X\times Y$ of a solution of \eqref{sec1:arrowhurwicz}, viz.,
\begin{equation*}
	(\sigma(t),\omega(t))=\frac{1}{t}\int_{0}^{t}(x(\tau),\lambda(\tau))\,
	\mathrm{d}\tau,
\end{equation*}
we find that the solutions of \eqref{sec1:arrowhurwicz} obey in average, for any $(\xi,\eta)\in S\times M$, the estimate
\begin{equation*}
	L(\sigma(t),\eta)-L(\xi,\omega(t))=\OO\Big(\frac{1}{t}\Big)\
	\text{as $t\to+\infty$}.
\end{equation*}
In this case, the Ces{\`a}ro average $(\sigma(t),\omega(t))$ of a solution of \eqref{sec1:arrowhurwicz} weakly converges, as $t\to+\infty$, towards a saddle point of $L$. This result is in line with the work by Nemirovski and Yudin \cite{ASN-DBY:78} on the classical Arrow--Hurwicz method and may also be deduced more elementary by the results of Baillon and Br{\'e}zis \cite{JBB-HB:76}.

Whenever $f$ is strongly convex, we obtain more stringent mini-maximizing properties of the solutions of \eqref{sec1:arrowhurwicz} relative to the primal-dual gap function. More precisely, we show that the solutions $(x,\lambda):[0,+\infty)\to X\times Y$ of \eqref{sec1:arrowhurwicz} evolve, for any $(\xi,\eta)\in S\times M$, according to the estimate
\begin{equation*}
	L(x(t),\eta)-L(\xi,\lambda(t))=o\Big(\frac{1}{\sqrt{t}}\Big)\
	\text{as $t\to+\infty$}.
\end{equation*}
Moreover, the solutions $(x(t),\lambda(t))$ of \eqref{sec1:arrowhurwicz} are proven to converge weakly, as $t\to+\infty$, towards an element of the set of saddle points of $L$. In particular, we characterize~the weak limit of a solution of \eqref{sec1:arrowhurwicz} as the orthogonal projection of its initial data $(x_{0},\lambda_{0})\in X\times Y$ onto the (closed and convex) set $S\times M$, i.e.,
\begin{equation*}
	(x(t),\lambda(t))\rightharpoonup\proj_{S\times M}(x_{0},\lambda_{0})\
	\text{as $t\to+\infty$}.
\end{equation*}
If, in addition, the linear operator $A^{\ast}$ is bounded from below, we observe that the so- lutions of \eqref{sec1:arrowhurwicz} obey, for $(\xi,\eta)\in S\times M$, the refined estimate
\begin{equation*}
	L(x(t),\eta)-L(\xi,\lambda(t))=o\Big(\frac{1}{t}\Big)\
	\text{as $t\to+\infty$}.
\end{equation*}
In this case, it is proven that the solutions $(x(t),\lambda(t))$ of \eqref{sec1:arrowhurwicz} strongly converge, as $t\to+\infty$, towards the unique saddle point of $L$.

Finally, we show that the solutions of \eqref{sec1:arrowhurwicz} decay asymptotically at an exponential rate provided that $f$ is twice continuously differentiable, satisfying
\begin{equation*}
	\iprodX{\nabla^{2}f(x)(x-y)}{x-y}\leq2D_{f}(y,x),\quad\forall x,y\in X.
\end{equation*}
Here, $D_{f}$ denotes the Bregman distance associated with $f$, cf. Bregman \cite{LMB:67}, and $\nabla^{2}f$ refers to the Hessian of $f$. In particular, we show that under the above condition there exists $\rho>0$ such that the solutions $(x,\lambda):[0,+\infty)\to X\times Y$ of \eqref{sec1:arrowhurwicz} verify, for any $(\xi,\eta)\in S\times M$, either one of the following exponential estimates:
\begin{align*}
	L(x(t),\eta)-L(\xi,\lambda(t))&=\OO\big(\mathrm{e}^{-2\rho t}\big)\
	\text{as $t\to+\infty$};\\
	L(x(t),\eta)-L(\xi,\lambda(t))&=\OO(t^{2}\mathrm{e}^{-2\rho t})\ \text{as $t\to+\infty$}.
\end{align*}
This result complements the decay rate estimates obtained earlier by Polyak \cite{BTP:70}.

\subsection{Organization}
We begin our discussion by reviewing some basic properties of the solutions of \eqref{sec1:arrowhurwicz} in the case of a convex objective function of the minimization problem \eqref{sec1:cvxproblem}. In Section~\ref{sec3}, we then investigate their asymptotic properties under the more stringent assumption of a strongly convex objective function. In Section \ref{sec4}, we show that the solutions of \eqref{sec1:arrowhurwicz} decay at an exponential rate provided that second-order information on the objective function is available. In Section \ref{sec5}, we highlight that the results on the Arrow--Hurwicz differential system \eqref{sec1:arrowhurwicz} may directly be conveyed to the case of solving structured minimization problems. Finally, Section \ref{sec6} is devoted to numerical experiments.

\section{Basic properties}\label{sec2}
Let $X\times Y$ be equipped with the Hilbertian product structure $\iprod{\,\cdot\,}{\,\cdot\,}=\iprodX{\,\cdot\,}{\,\cdot\,}+\iprodY{\,\cdot\,}{\,\cdot\,}$ and induced norm $\norm{\,\cdot\,}$. Throughout the text, we assume that
\begin{enumerate}[\hspace{6pt}({A}1)]
	\item $f:X\to\rl$ is convex and continuously differentiable;
	\item $\nabla f:X\to X$ is Lipschitz continuous on bounded sets;
	\item $A:X\to Y$ is linear and continuous, and $b\in Y$.
\end{enumerate}

Consider the Arrow--Hurwicz differential system
\begin{equation}
\renewcommand{\theequation}{AH}\tag{\theequation}\label{sec2:arrowhurwicz}
	\begin{cases}
		\dot{x}+\nabla f(x)+A^{\ast}\lambda=0_{X}\\
		\hspace{30pt}\dot{\lambda}+b-Ax=0_{Y}
	\end{cases}
\end{equation}
with initial data $(x_{0},\lambda_{0})\in X\times Y$, and recall that $(x,\lambda):[0,+\infty)\to X\times Y$ is a (classical) solution of \eqref{sec2:arrowhurwicz} if $(x,\lambda)\in\CC^{1}([0,+\infty);X\times Y)$ and $(x,\lambda)$ satisfies \eqref{sec2:arrowhurwicz} on $[0,+\infty)$ with $(x(0),\lambda(0))=(x_{0},\lambda_{0})$. The following result is an immediate consequence of the monotonicity of the operator
\begin{align*}
	T:X\times Y&\longrightarrow X\times Y\\
	(x,\lambda)&\longmapsto(\nabla f(x)+A^{\ast}\lambda,b-Ax)
\end{align*}
associated with the \eqref{sec2:arrowhurwicz} differential system; cf. Br{\'e}zis \cite[Theorem 3.1]{HB:73}, Aubin and~Cel- lina \cite[Theorem 3.2.1]{JPA-AC:84}. The existence and uniqueness of the (classical) solutions of \eqref{sec1:arrowhurwicz} thereby follow at once from the Cauchy--Lipschitz theorem\footnote{%
	Given the above assumptions, it is easy to verify that $(x,\lambda)\mapsto T(x,\lambda)$ is Lipschitz continuous on the bounded subsets of $X\times Y$.}; see, e.g., Haraux \cite[Proposition 6.2.1]{AH:91}.
\begin{theorem}\label{sec2:th:existence}
	For any $(x_{0},\lambda_{0})\in X\times Y$, there exists a unique solution $(x,\lambda):[0,+\infty)\to X\times Y$ of \eqref{sec2:arrowhurwicz}. Moreover,
	\begin{enumerate}[(i)]
		\item $t\mapsto\norm{(\dot{x}(t),\dot{\lambda}(t))}$ is non-increasing and
		\begin{equation*}
			\norm{(\dot{x}(t),\dot{\lambda}(t))}\leq\norm{T(x_{0},\lambda_{0})},
			\quad\forall t\geq0;
		\end{equation*}
		\item $\lim_{t\to+\infty}\norm{(\dot{x}(t),\dot{\lambda}(t))}$ exists.
	\end{enumerate}
\end{theorem}
\begin{remark}
	We note that the assertions of Theorem \ref{sec2:th:existence} essentially remain valid even under the assumption that $f:X\to\rl\cup\{+\infty\}$ is a proper convex lower semi-continuous function. In this case, the \eqref{sec2:arrowhurwicz} dynamics generalize to the evolution system
	\begin{equation*}
		\begin{cases}
			\dot{x}+\partial f(x)+A^{\ast}\lambda\ni0_{X}\\
			\hspace{26.5pt}\dot{\lambda}+b-Ax=0_{Y}
		\end{cases}
	\end{equation*}
	with $\partial f$ denoting the convex subdifferential of $f$. The existence and uniqueness of the (strong) solutions of the above differential system are then deduced from the general theory for semi-groups of contractions generated by maximally monotone operators; cf. Br{\'e}zis \cite{HB:73} (see also Pazy \cite{AP:79}, Peypouquet and Sorin \cite{JP-SS:10}).
\end{remark}

In the following, let $S\times M$ denote the set of saddle points of the Lagrangian
\begin{align*}
	L:X\times Y&\longrightarrow\rl\\
	(x,\lambda)&\longmapsto f(x)+\iprodY{\lambda}{Ax-b}
\end{align*}
associated with the convex minimization problem \eqref{sec1:cvxproblem}. Using the convexity of $f$, we immediately observe that, for any $(x,\lambda),(y,\eta)\in X\times Y$, it holds that
\begin{equation}\label{sec2:lagrangianbound}
	\iprod{T(x,\lambda)}{(x,\lambda)-(y,\eta)}\geq L(x,\eta)-L(y,\lambda).
\end{equation}
Anchoring the above inequality to the set $S\times M$ yields the following integrability result for the primal-dual gap function.
\begin{proposition}\label{sec2:pr:lagrangian}
	Let $S\times M$ be non-empty and let $(x,\lambda):[0,+\infty)\to X\times Y$ be a solution of \eqref{sec2:arrowhurwicz}. Then, for any $(\xi,\eta)\in S\times M$,
	\begin{enumerate}[(i)]
		\item $\lim_{t\to+\infty}\norm{(x(t),\lambda(t))-(\xi,\eta)}$ exists;
		\item it holds that
		\begin{equation*}
			\int_{0}^{\infty}L(x(\tau),\eta)-L(\xi,\lambda(\tau))\,\mathrm{d}\tau
			<+\infty.
		\end{equation*}
	\end{enumerate}
\end{proposition}
\begin{proof}
	{\it (i)} Let $(\xi,\eta)\in S\times M$. Using \eqref{sec2:arrowhurwicz} together with \eqref{sec2:lagrangianbound}, we have for any $t\geq0$,
	\begin{equation}\label{sec2:lagdec}
		\frac{\mathrm{d}}{\mathrm{d}t}\big(\normX{x(t)-\xi}^{2}+
		\normY{\lambda(t)-\eta}^{2}\big)/2+L(x(t),\eta)-L(\xi,\lambda(t))\leq0.
	\end{equation}
	Since $(\xi,\eta)$ belongs to $S\times M$, it follows that $L(x(t),\eta)-L(\xi,\eta)\geq0$ as well as $L(\xi,\eta)-L(\xi,\lambda(t))\geq0$ and thus,
	\begin{equation*}
		L(x(t),\eta)-L(\xi,\lambda(t))\geq0.
	\end{equation*}
	Consequently, $t\mapsto\norm{(x(t),\lambda(t))-(\xi,\eta)}$ is non-increasing and bounded from below on $[0,+\infty)$ so that
	\begin{equation*}
		\lim_{t\to+\infty}\norm{(x(t),\lambda(t))-(\xi,\eta)}\ \text{exists}.
	\end{equation*}

	{\it (ii)} Integration of \eqref{sec2:lagdec} over $[0,t]$ yields
	\begin{equation}\label{sec2:lagdecint}
		\begin{split}
			\normX{x(t)-\xi}^{2}/2+\normY{\lambda(t)-\eta}^{2}/2&+
			\int_{0}^{t}L(x(\tau),\eta)-L(\xi,\lambda(\tau))\,\mathrm{d}\tau\\
			\leq\normX{x(0)-\xi}^{2}/2&+\normY{\lambda(0)-\eta}^{2}/2.
		\end{split}
	\end{equation}
	Taking into account that $\norm{(x(t),\lambda(t))-(\xi,\eta)}^{2}\geq0$, we obtain
	\begin{equation*}
		\int_{0}^{t}L(x(\tau),\eta)-L(\xi,\lambda(\tau))\,\mathrm{d}\tau\leq
		\normX{x(0)-\xi}^{2}/2+\normY{\lambda(0)-\eta}^{2}/2.
	\end{equation*}
	This majorization being valid for any $t\geq0$, taking the supremum gives
	\begin{equation*}
		\int_{0}^{\infty}L(x(\tau),\eta)-L(\xi,\lambda(\tau))\,\mathrm{d}\tau<+
		\infty.\qedhere
	\end{equation*}
\end{proof}
\begin{remark}
	Proposition \ref{sec2:pr:lagrangian}{\it (i)} asserts that the solutions of \eqref{sec2:arrowhurwicz} remain bounded whenever the set $S\times M$ is non-empty. Conversely, it can be shown that the set $S\times M$ is non-empty whenever \eqref{sec2:arrowhurwicz} admits a bounded solution; cf. Pazy \cite{AP:79}.
\end{remark}

Let us next focus on the asymptotic properties of the solutions of \eqref{sec2:arrowhurwicz}. To this end, define the Ces\`{a}ro average $(\sigma,\omega):(0,+\infty)\to X\times Y$ of a solution of \eqref{sec2:arrowhurwicz} by
\begin{equation*}
	(\sigma(t),\omega(t))=\frac{1}{t}\int_{0}^{t}(x(\tau),\lambda(\tau))\,
	\mathrm{d}\tau.
\end{equation*}
The following result asserts the weak convergence of the Ces\`{a}ro average of a solution of \eqref{sec2:arrowhurwicz} and further provides an estimate on the decay of the primal-dual gap function.
\begin{proposition}\label{sec2:pr:weakergodic}
	Let $S\times M$ be non-empty and let $(\sigma,\omega):(0,+\infty)\to X\times Y$ be the Ces\`{a}ro average of a solution of \eqref{sec2:arrowhurwicz}. Then, for any $(\xi,\eta)\in S\times M$, it holds that
	\begin{equation*}
		L(\sigma(t),\eta)-L(\xi,\omega(t))=\OO\Big(\frac{1}{t}\Big)\
		\text{as $t\to+\infty$}.
	\end{equation*}
	Moreover, there exists $(\bar{x},\bar{\lambda})\in S\times M$ such that $(\sigma(t),\omega(t))\rightharpoonup(\bar{x},\bar{\lambda})$ weakly in $X\times Y$ as $t\to+\infty$.
\end{proposition}
\begin{proof}
	Let $(\xi,\eta)\in S\times M$ and recall from \eqref{sec2:lagdecint} that for any $t\geq0$,
	\begin{align*}
		\normX{x(t)-\xi}^{2}/2+\normY{\lambda(t)-\eta}^{2}/2&+
		\int_{0}^{t}L(x(\tau),\eta)-L(\xi,\lambda(\tau))\,\mathrm{d}\tau\\
		\leq\normX{x(0)-\xi}^{2}/2&+\normY{\lambda(0)-\eta}^{2}/2.
	\end{align*}
	Dividing the above inequality by $t>0$ and taking into account that $\norm{(x(t),\lambda(t))-(\xi,\eta)}^{2}\geq0$, we obtain
	\begin{equation*}
		\frac{1}{t}\int_{0}^{t}L(x(\tau),\eta)-L(\xi,\lambda(\tau))\,\mathrm{d}\tau
		\leq\frac{1}{2t}\big(\normX{x(0)-\xi}^{2}+\normY{\lambda(0)-\eta}^{2}\big).
	\end{equation*}
	By Jensen's inequality, as $L(\,\cdot\,,\eta)$ and $-L(\xi,\,\cdot\,)$ are both convex, it follows
	\begin{equation*}
		L(\sigma(t),\eta)-L(\xi,\omega(t))\leq
		\frac{1}{2t}\big(\normX{x(0)-\xi}^{2}+\normY{\lambda(0)-\eta}^{2}\big)
	\end{equation*}
	and thus,
	\begin{equation*}
		\limsup_{t\to+\infty}\,t\big(L(\sigma(t),\eta)-L(\xi,\omega(t))\big)
		<+\infty.
	\end{equation*}

	The weak convergence of $(\sigma(t),\omega(t))$ as $t\to+\infty$ is an immediate consequence of the Opial--Passty lemma applied to the set $S\times M$; cf. Passty \cite{GBP:79}.
\end{proof}
\begin{remark}
	We note that the weak ergodic convergence of the solutions of \eqref{sec1:arrowhurwicz} may also be deduced from the general theory for semi-groups of contractions generated by maximally monotone operators; cf. Baillon and Br{\'e}zis \cite{JBB-HB:76} (see also Br{\'e}zis \cite{HB:78} for a localization result of the weak limit).
\end{remark}

\section{The strongly monotone case}\label{sec3}
In this section, we investigate the asymptotic properties of the solutions of \eqref{sec2:arrowhurwicz} under the more stringent assumptions that
\begin{enumerate}[\hspace{6pt}({A}1)]\setcounter{enumi}{3}
	\item $\nabla f:X\to X$ is $\alpha$-strongly monotone, i.e.,
	\begin{equation*}
		\exists\alpha>0\ \forall x,y\in X\quad
		\iprodX{\nabla f(x)-\nabla f(y)}{x-y}\geq\alpha\normX{x-y}^{2};
	\end{equation*}
	\item $S\times M\subset X\times Y$ is non-empty.
\end{enumerate}
We recall that the latter assumption is verified whenever \eqref{sec1:cvxproblem} admits a minimizer and, for instance, the constraint qualification
\begin{equation*}
	b\in\sri A(X)
\end{equation*}
holds; see, e.g., Bauschke and Combettes \cite{HHB-PLC:17}. In turn, the former assumption implies that $S$ is reduced to a singleton. 

\subsection{Weak convergence}
Let us begin with a result on the weak convergence of the solutions of \eqref{sec1:arrowhurwicz}. Since $f$ is $\alpha$-strongly convex ($\nabla f$ being $\alpha$-strongly monotone), we have for any $(x,\lambda),(y,\eta)\in X\times Y$,
\begin{equation}\label{sec3:lagrangianbound}
	\iprod{T(x,\lambda)}{(x,\lambda)-(y,\eta)}\geq L(x,\eta)-L(y,\lambda)+
	\alpha\normX{x-y}^{2}/2.
\end{equation}
Utilizing this inequality relative to the set $S\times M$ gives the following asymptotic properties of the solutions of \eqref{sec2:arrowhurwicz} and the primal-dual gap function.
\begin{theorem}\label{sec3:th:weakconv}
	Let $\nabla f:X\to X$ be $\alpha$-strongly monotone, let $S\times M$ be non-empty, and let $(x,\lambda):[0,+\infty)\to X\times Y$ be a solution of \eqref{sec2:arrowhurwicz}. Then, for any $(\xi,\eta)\in S\times M$, it holds that
	\begin{align*}
		L(x(t),\eta)-L(\xi,\lambda(t))&=o\Big(\frac{1}{\sqrt{t}}\Big)\
		\text{as $t\to+\infty$};\\
		\norm{(\dot{x}(t),\dot{\lambda}(t))}&=o\Big(\frac{1}{\sqrt{t}}\Big)\
		\text{as $t\to+\infty$}.
	\end{align*}
	Moreover, there exists $(\bar{x},\bar{\lambda})\in S\times M$ such that $(x(t),\lambda(t))\rightharpoonup(\bar{x},\bar{\lambda})$ weakly in $X\times Y$ as $t\to+\infty$.
\end{theorem}
\begin{proof}
	Let $(\xi,\eta)\in S\times M$. Using the `Lagrangian identity' \eqref{sec1:lagidentity}, we have for any $t\geq0$,
	\begin{equation*}
		\frac{\mathrm{d}}{\mathrm{d}t}\big(L(\xi,\eta)-L(x(t),\lambda(t))\big)+
		\normX{\dot{x}(t)}^{2}+\normY{\dot{\lambda}(t)}^{2}=2\normX{\dot{x}(t)}^{2}.
	\end{equation*}
	Integration over $[0,t]$ yields
	\begin{align*}
		L(\xi,\eta)-L(x(t),\lambda(t))&+\int_{0}^{t}\normX{\dot{x}(\tau)}^{2}+
		\normY{\dot{\lambda}(\tau)}^{2}\,\mathrm{d}\tau\\
		=2\int_{0}^{t}\normX{\dot{x}(\tau)}^{2}\,\mathrm{d}\tau
		&+L(\xi,\eta)-L(x(0),\lambda(0)).
	\end{align*}
	Since $\nabla f$ is $\alpha$-strongly monotone, we readily deduce from Theorem \ref{sec2:th:existence}{\it (i)} that
	\begin{equation}\label{sec3:strongmono}
		\normX{\dot{x}(t)}^{2}/2+\normY{\dot{\lambda}(t)}^{2}/2
		+\alpha\int_{0}^{t}\normX{\dot{x}(\tau)}^{2}\,\mathrm{d}\tau
		\leq\normX{\dot{x}(0)}^{2}/2+\normY{\dot{\lambda}(0)}^{2}/2.
	\end{equation}
	Using this inequality together with the above equation, we obtain
	\begin{equation*}
		L(\xi,\eta)-L(x(t),\lambda(t))+\normX{\dot{x}(t)}^{2}/\alpha+
		\normY{\dot{\lambda}(t)}^{2}/\alpha+\int_{0}^{t}\normX{\dot{x}(\tau)}^{2}+
		\normY{\dot{\lambda}(\tau)}^{2}\,\mathrm{d}\tau\leq C,
	\end{equation*}
	where $C=L(\xi,\eta)-L(x(0),\lambda(0))+\normX{\dot{x}(0)}^{2}/\alpha+\normY{\dot{\lambda}(0)}^{2}/\alpha$. Moreover, from \eqref{sec2:arrowhurwicz} and the fact that $T(\xi,\eta)=(0_{X},0_{Y})$, we infer
	\begin{equation}
		\begin{split}\label{sec3:lagdifequality}
			L(\xi,\eta)-L(x(t),\lambda(t))&=f(\xi)-f(x(t))+
			\iprodX{\dot{x}(t)+\nabla f(x(t))}{x(t)-\xi}\\
			&=\frac{\mathrm{d}}{\mathrm{d}t}\normX{x(t)-\xi}^{2}/2+D_{f}(\xi,x(t))
		\end{split}
	\end{equation}
	with $D_{f}$ denoting the Bregman distance associated with $f$. Since $f$ is $\alpha$-strongly convex, we have $D_{f}(\xi,x(t))\geq\alpha\normX{x(t)-\xi}^{2}/2$ and thus,
	\begin{equation*}
		L(\xi,\eta)-L(x(t),\lambda(t))\geq\frac{\mathrm{d}}{\mathrm{d}t}
		\normX{x(t)-\xi}^{2}/2+\alpha\normX{x(t)-\xi}^{2}/2.
	\end{equation*}
	Using this inequality together with the fact that
	\begin{equation*}
		\alpha\normX{x(t)-\xi}^{2}/2-\alpha\normX{x(0)-\xi}^{2}/2=
		\alpha\int_{0}^{t}\frac{\mathrm{d}}{\mathrm{d}\tau}
		\normX{x(\tau)-\xi}^{2}/2\,\mathrm{d}\tau,
	\end{equation*}
	we obtain
	\begin{align*}
		L(\xi,\eta)-L(x(t),\lambda(t))&\geq\frac{\mathrm{d}}{\mathrm{d}t}
		\normX{x(t)-\xi}^{2}/2+\alpha\int_{0}^{t}
		\frac{\mathrm{d}}{\mathrm{d}\tau}\normX{x(\tau)-\xi}^{2}/2\,\mathrm{d}\tau\\
		&\quad+\alpha\normX{x(0)-\xi}^{2}/2.
	\end{align*}
	In view of the above derivations, we deduce
	\begin{align*}
		\frac{\mathrm{d}}{\mathrm{d}t}\normX{x(t)-\xi}^{2}/2+
		\alpha\int_{0}^{t}\frac{\mathrm{d}}{\mathrm{d}\tau}
		\normX{x(\tau)-\xi}^{2}/2\,\mathrm{d}\tau&+\normX{\dot{x}(t)}^{2}/\alpha+
		\normY{\dot{\lambda}(t)}^{2}/\alpha\\
		+\,\alpha\int_{0}^{t}\normX{\dot{x}(\tau)}^{2}/\alpha+
		\normY{\dot{\lambda}(\tau)}^{2}/\alpha\,\mathrm{d}\tau\leq
		C&-\alpha\normX{x(0)-\xi}^{2}/2.
	\end{align*}
	Multiplying the above inequality by $\mathrm{e}^{\alpha t}$ and integrating over $[0,t]$ yields
	\begin{equation*}
		\alpha\normX{x(t)-\xi}^{2}/2+\int_{0}^{t}\normX{\dot{x}(\tau)}^{2}+
		\normY{\dot{\lambda}(\tau)}^{2}\,\mathrm{d}\tau\leq\tilde{C}
	\end{equation*}
	for some sufficiently large constant $\tilde{C}$. Taking into account that $\normX{x(t)-\xi}^{2}\geq0$ and subsequently passing to the limit as $t\to+\infty$ gives
	\begin{equation*}
		\int_{0}^{\infty}\normX{\dot{x}(\tau)}^{2}+
		\normY{\dot{\lambda}(\tau)}^{2}\,\mathrm{d}\tau<+\infty.
	\end{equation*}
	Moreover, since $t\mapsto\norm{(\dot{x}(t),\dot{\lambda}(t))}^{2}$ is non-increasing on $[0,+\infty)$, cf. Theorem \ref{sec2:th:existence}{\it (i)}, we have for any $t\geq0$,
	\begin{equation*}
		\int_{t/2}^{t}\normX{\dot{x}(\tau)}^{2}+\normY{\dot{\lambda}(\tau)}^{2}\,
		\mathrm{d}\tau\geq t\big(\normX{\dot{x}(t)}^{2}+
		\normY{\dot{\lambda}(t)}^{2}\big)/2.
	\end{equation*}
	Observing that $\norm{(\dot{x},\dot{\lambda})}^{2}$ belongs to $\LL^{1}([0,+\infty);\rl)$ entails
	\begin{equation*}
		\lim_{t\to+\infty}t\big(\normX{\dot{x}(t)}^{2}+\normY{\dot{\lambda}(t)}^{2}
		\big)=0.
	\end{equation*}
	Finally, from inequality \eqref{sec3:lagrangianbound}, we have for any $t\geq0$,
	\begin{equation*}
		L(x(t),\eta)-L(\xi,\lambda(t))+\alpha\normX{x(t)-\xi}^{2}/2
		\leq\iprod{T(x(t),\lambda(t))}{(x(t),\lambda(t))-(\xi,\eta)}.
	\end{equation*}
	In view of \eqref{sec2:arrowhurwicz} and the Cauchy--Schwarz inequality, we get
	\begin{equation}\label{sec3:cauchyschwarz}
		L(x(t),\eta)-L(\xi,\lambda(t))+\alpha\normX{x(t)-\xi}^{2}/2
		\leq\norm{(x(t),\lambda(t))-(\xi,\eta)}\norm{(\dot{x}(t),\dot{\lambda}(t))}.
	\end{equation}
	Using that $(x,\lambda)$ remains bounded on $[0,+\infty)$, cf. Proposition \ref{sec2:pr:lagrangian}{\it (i)}, there exists $\hat{C}\geq0$ such that
	\begin{equation*}
		L(x(t),\eta)-L(\xi,\lambda(t))+\alpha\normX{x(t)-\xi}^{2}/2\leq
		\hat{C}\norm{(\dot{x}(t),\dot{\lambda}(t))}.
	\end{equation*}
	Multiplying the above inequality by $\sqrt{t}$ gives
	\begin{equation*}
		\sqrt{t}\big(L(x(t),\eta)-L(\xi,\lambda(t))+\alpha\normX{x(t)-\xi}^{2}/2
		\big)\leq\hat{C}\sqrt{t}\norm{(\dot{x}(t),\dot{\lambda}(t))}.
	\end{equation*}
	Observing that $\lim_{t\to+\infty}\sqrt{t}\norm{(\dot{x}(t),\dot{\lambda}(t))}=0$, passing to the limit as $t\to+\infty$ yields
	\begin{equation*}
		\lim_{t\to+\infty}\sqrt{t}\big(L(x(t),\eta)-L(\xi,\lambda(t))+
		\alpha\normX{x(t)-\xi}^{2}/2\big)=0.
	\end{equation*}

	The weak convergence of $(x(t),\lambda(t))$ as $t\to+\infty$ is an immediate consequence of the Opial lemma applied to the set $S\times M$; cf. Opial \cite{ZO:67}.
\end{proof}
\begin{remark}
	We note that the weak convergence of the solutions of \eqref{sec1:arrowhurwicz} may also be deduced from the graph closedness property of the maximally monotone operator $T$ with respect to the weak-strong topology; see, e.g., Bauschke and Combettes \cite{HHB-PLC:17}. Yet another tool to establish the weak convergence of the solutions of \eqref{sec1:arrowhurwicz} is the concept of demipositivity, first developed by Bruck \cite{REB:75} for monotone operators, and later extended by Chbani and Riahi \cite{ZC-HR:14} to monotone bifunctions. However, the maximally monotone operator $T$ associated with the Lagrangian $L$ of the convex minimization problem \eqref{sec1:cvxproblem} need, in general, not be demipositive. We leave the details to the reader.
\end{remark}

To further localize the weak limit of a solution of \eqref{sec2:arrowhurwicz}, recall that the set $S\times M$ (if non-empty) is of the form $\{\bar{x}\}\times M$, where $\bar{x}\in X$ is the unique minimizer of \eqref{sec1:cvxproblem} and $M\subset Y$ refers to the closed affine subspace of Lagrange multipliers, viz.,
\begin{equation*}
	M=\{\bar{\lambda}\in Y\mid \nabla f(\bar{x})+A^{\ast}\bar{\lambda}=0_{X}\}.
\end{equation*}
The following result characterizes the weak limit of a solution of \eqref{sec2:arrowhurwicz} as the orthogonal projection of its initial data onto the (closed and convex) set $S\times M$.
\begin{corollary}
	Under the hypotheses of Theorem \ref{sec3:th:weakconv}, let $(\bar{x},\bar{\lambda})\in S\times M$ be such that $(x(t),\lambda(t))\rightharpoonup(\bar{x},\bar{\lambda})$ weakly in $X\times Y$ as $t\to+\infty$. Then, $(\bar{x},\bar{\lambda})=\proj_{S\times M}(x_{0},\lambda_{0})$.
\end{corollary}
\begin{proof}
	Let $(\bar{x},\bar{\lambda})\in S\times M$ be such that $(x(t),\lambda(t))\rightharpoonup(\bar{x},\bar{\lambda})$ weakly in $X\times Y$ as $t\to+\infty$~and let $(\xi,\eta)\in S\times M$ be arbitrary. Using \eqref{sec2:arrowhurwicz} and the fact that $T(\xi,\eta)=(0_{X},0_{Y})$, we have for any $t\geq0$,
	\begin{equation*}
		\iprod{(\bar{x},\bar{\lambda})-(\xi,\eta)}{(\dot{x}(t),\dot{\lambda}(t))}+
		\iprod{T(x(t),\lambda(t))-T(\xi,\eta)}{(\bar{x},\bar{\lambda})-(\xi,\eta)}=
		0.
	\end{equation*}
	Observing that
	\begin{align*}
		\iprod{T(x(t),\lambda(t))-T(\xi,\eta)}{(\bar{x},\bar{\lambda})&-(\xi,\eta)}=
		\iprodX{\nabla f(x(t))-\nabla f(\xi)}{\bar{x}-\xi}\\
		+\,\iprodX{\nabla f(\bar{x})&-\nabla f(\xi)}{x(t)-\xi},
	\end{align*}
	and noticing that the right-hand side of the above equation vanishes (as $S$ is reduced to a singleton), it follows that
	\begin{equation*}
		\iprod{(\bar{x},\bar{\lambda})-(\xi,\eta)}{(\dot{x}(t),\dot{\lambda}(t))}=0.
	\end{equation*}
	Integration over $[0,t]$ yields
	\begin{equation*}
		\iprod{(\bar{x},\bar{\lambda})-(\xi,\eta)}
		{(x(t),\lambda(t))-(x(0),\lambda(0))}=0.
	\end{equation*}
	Since $(x(t),\lambda(t))\rightharpoonup(\bar{x},\bar{\lambda})$ weakly in $X\times Y$ as $t\to+\infty$, we infer
	\begin{equation*}
		\iprod{(\bar{x},\bar{\lambda})-(\xi,\eta)}{(\bar{x},\bar{\lambda})-
			(x(0),\lambda(0))}=0.
	\end{equation*}
	The above equality being true for any $(\xi,\eta)\in S\times M$, we conclude by virtue of the pro- jection theorem; see, e.g., Bauschke and Combettes \cite{HHB-PLC:17}.
\end{proof}

Finally, as an immediate consequence of Theorem \ref{sec3:th:weakconv}, we have the following refined asymptotic estimates whenever $A:X\to Y$ is bounded from below\footnote{%
	We recall that $A:X\to Y$ is bounded from below if and only if it is injective with closed range; see, e.g., Br{\'e}zis \cite{HB:11}.
}, i.e.,
\begin{equation*}\label{sec3:injective}
	\exists\beta>0\ \forall x\in X\quad\normY{Ax}^{2}\geq\beta\normX{x}^{2}.
\end{equation*}
\begin{corollary}\label{sec3:co:refinedgap}
	Under the hypotheses of Theorem \ref{sec3:th:weakconv}, let $A:X\to Y$ be bounded from below. Then, for any $(\xi,\eta)\in S\times M$, it holds that
	\begin{align*}
		L(x(t),\eta)-L(\xi,\lambda(t))&=o\Big(\frac{1}{t}\Big)\
		\text{as $t\to+\infty$};\\
		\normX{x(t)-\xi}&=o\Big(\frac{1}{\sqrt{t}}\Big)\ \text{as $t\to+\infty$}.
	\end{align*}
\end{corollary}
\begin{proof}
	Let $(\xi,\eta)\in S\times M$. Since $A$ is bounded from below, there exists $\beta>0$ such that for any $t\geq0$,
	\begin{equation*}
		\normX{\dot{x}(t)}^{2}+\beta\normX{x(t)-\xi}^{2}\leq
		\normX{\dot{x}(t)}^{2}+\normY{A(x(t)-\xi)}^{2}.
	\end{equation*}
	Using \eqref{sec2:arrowhurwicz} together with the fact that $A\xi-b=0_{Y}$, we obtain
	\begin{equation*}
		\normX{\dot{x}(t)}^{2}+\beta\normX{x(t)-\xi}^{2}\leq
		\normX{\dot{x}(t)}^{2}+\normY{\dot{\lambda}(t)}^{2}.
	\end{equation*}
	Multiplying the above inequality by $t$ and using $\lim_{t\to+\infty}t\norm{(\dot{x}(t),\dot{\lambda}(t))}^{2}=0$, cf. Theorem \ref{sec3:th:weakconv}, we infer
	\begin{equation*}
		\lim_{t\to+\infty}t\big(\normX{\dot{x}(t)}^{2}+
		\beta\normX{x(t)-\xi}^{2}\big)=0.
	\end{equation*}
	Moreover, from \eqref{sec3:lagrangianbound} together with $T(\xi,\eta)=(0_{X},0_{Y})$, we observe that for any $t\geq0$,
	\begin{align*}
		L(x(t),\eta)-L(\xi,\lambda(t))+\alpha\normX{x(t)-\xi}^{2}/2
		&\leq\iprod{T(x(t),\lambda(t))-T(\xi,\eta)}{(x(t),\lambda(t))-(\xi,\eta)}\\
		&=\iprodX{\nabla f(x(t))-\nabla f(\xi)}{x(t)-\xi}.
	\end{align*}
	Using that $\nabla f$ is Lipschitz continuous on bounded sets, there exists $\gamma\geq0$ such that
	\begin{equation*}
		L(x(t),\eta)-L(\xi,\lambda(t))+\alpha\normX{x(t)-\xi}^{2}/2\leq
		\gamma\normX{x(t)-\xi}^{2}.
	\end{equation*}
	Multiplying this inequality by $t$ and using $\lim_{t\to+\infty}t\normX{x(t)-\xi}^{2}=0$, we conclude
	\begin{equation*}
		\lim_{t\to+\infty}t\big(L(x(t),\eta)-L(\xi,\lambda(t))+
		\alpha\normX{x(t)-\xi}^{2}/2\big)=0.\qedhere
	\end{equation*}
\end{proof}

\subsection{Strong convergence}
Let us now complement the previous discussion with a result on the strong convergence of the solutions of \eqref{sec2:arrowhurwicz}. To this end, we assume that $A^{\ast}:Y\to X$ is bounded from below\footnote{%
	We note that $A^{\ast}:Y\to X$ is bounded from below if and only if $A$ is surjective; see, e.g., Br{\'e}zis \cite{HB:11}.
}, i.e.,
\begin{equation*}
	\exists\beta>0\ \forall y\in Y\quad
	\normX{A^{\ast}y}^{2}\geq\beta\normY{y}^{2}.
\end{equation*}
This clearly implies that the set $S\times M$ is reduced to $\{(\bar{x},\bar{\lambda})\}$, where $\bar{x}\in X$ is the unique minimizer of \eqref{sec1:cvxproblem} and $\bar{\lambda}\in Y$ refers to the corresponding Lagrange multiplier given by
\begin{equation*}
	\bar{\lambda}=-(AA^{\ast})^{-1}A\nabla f(\bar{x}).
\end{equation*}
\begin{proposition}\label{sec3:pr:strongconv}
	Let $\nabla f:X\to X$ be $\alpha$-strongly monotone, let $A^{\ast}:Y\to X$ be bounded from below, and let $(x,\lambda):[0,+\infty)\to X\times Y$ be a solution of \eqref{sec2:arrowhurwicz}. Then, for $(\xi,\eta)\in S\times M$, it holds that
	\begin{align*}
		L(x(t),\eta)-L(\xi,\lambda(t))&=o\Big(\frac{1}{t}\Big)\
		\text{as $t\to+\infty$};\\
		\norm{(x(t),\lambda(t))-(\xi,\eta)}&=
		o\Big(\frac{1}{\sqrt{t}}\Big)\ \text{as $t\to+\infty$}.
	\end{align*}
	Consequently, $(x(t),\lambda(t))$ converges strongly, as $t\to+\infty$, to the unique element in $S\times M$.
\end{proposition}
\begin{proof}
	Let $(\xi,\eta)$ be the unique element in $S\times M$. Using \eqref{sec2:arrowhurwicz} together with $T(\xi,\eta)=(0_{X},0_{Y})$ and the fact that $\nabla f$ is $\alpha$-strongly monotone, we have for any $t\geq0$,
	\begin{equation}
		\begin{split}\label{sec3:strongmonestimate}
			\frac{\mathrm{d}}{\mathrm{d}t}\big(\normX{x(t)-\xi}^{2}+
			\normY{\lambda(t)-\eta}^{2}\big)/2&+\alpha
			\big(\normX{x(t)-\xi}^{2}+\normY{\lambda(t)-\eta}^{2}\big)\\
			\leq\alpha\normY{\lambda(t)&-\eta}^{2}.
		\end{split}
	\end{equation}
	Since $A^{\ast}$ is bounded from below, there exists $\beta>0$ such that
	\begin{equation*}
		\beta\normY{\lambda(t)-\eta}^{2}\leq
		\normX{A^{\ast}(\lambda(t)-\eta)}^{2}.
	\end{equation*}
	Using again \eqref{sec2:arrowhurwicz} together with $\nabla f(\xi)+A^{\ast}\eta=0_{X}$, we get
	\begin{equation*}
		\beta\normY{\lambda(t)-\eta}^{2}\leq
		\normX{\dot{x}(t)+\nabla f(x(t))-\nabla f(\xi)}^{2}
	\end{equation*}
	and thus,
	\begin{equation*}
		\beta\normY{\lambda(t)-\eta}^{2}/2\leq\normX{\dot{x}(t)}^{2}+
		\normX{\nabla f(x(t))-\nabla f(\xi)}^{2}.
	\end{equation*}
	Since $(x,\lambda)$ remains bounded on $[0,+\infty)$, cf. Proposition \ref{sec2:pr:lagrangian}{\it (i)}, and owing to the fact that $\nabla f$ is Lipschitz continuous on bounded sets, there further exists $\gamma\geq0$ such that
	\begin{equation*}
		\beta\normY{\lambda(t)-\eta}^{2}/2\leq\normX{\dot{x}(t)}^{2}+
		\gamma^{2}\normX{x(t)-\xi}^{2}.
	\end{equation*}
	In view of the above derivations, we obtain
	\begin{align*}
		\frac{\mathrm{d}}{\mathrm{d}t}\big(\normX{x(t)-\xi}^{2}+
		\normY{\lambda(t)-\eta}^{2}\big)/2+\alpha
		\big(\normX{x(t)&-\xi}^{2}+\normY{\lambda(t)-\eta}^{2}\big)\\
		\leq2\alpha\normX{\dot{x}(t)}^{2}/\beta+
		2\alpha\gamma^{2}\normX{x(t)&-\xi}^{2}/\beta,
	\end{align*}
	which, by applying \eqref{sec3:strongmonestimate} again, reads
	\begin{align*}
		(&\beta/2+\gamma^{2})\frac{\mathrm{d}}{\mathrm{d}t}
		\big(\normX{x(t)-\xi}^{2}+\normY{\lambda(t)-\eta}^{2}\big)\\
		+\,\alpha&\beta\big(\normX{x(t)-\xi}^{2}+
		\normY{\lambda(t)-\eta}^{2}\big)\leq2\alpha\normX{\dot{x}(t)}^{2}.
	\end{align*}
	Integration over $[0,t]$ yields
	\begin{align*}
		(\beta/2+\gamma^{2})\big(\normX{x(t)-\xi}^{2}+
		\normY{\lambda(t)-\eta}^{2}\big)+
		\alpha\beta\int_{0}^{t}\normX{x(\tau)&-\xi}^{2}+
		\normY{\lambda(\tau)-\eta}^{2}\,\mathrm{d}\tau\\
		\leq2\alpha\int_{0}^{t}\normX{\dot{x}(\tau)}^{2}\,\mathrm{d}\tau+
		(\beta/2+\gamma^{2})\big(\normX{x(0)-\xi}^{2}&+
		\normY{\lambda(0)-\eta}^{2}\big).
	\end{align*}
	Combining the above inequality with \eqref{sec3:strongmono} gives
	\begin{align*}
		\normX{\dot{x}(t)}^{2}+\normY{\dot{\lambda}(t)}^{2}+(\beta/2+\gamma^{2})
		\big(\normX{x(t)&-\xi}^{2}+\normY{\lambda(t)-\eta}^{2}\big)\\
		+\,\alpha\beta\int_{0}^{t}\normX{x(\tau)-\xi}^{2}+
		\normY{\lambda(\tau)&-\eta}^{2}\,\mathrm{d}\tau\leq C,
	\end{align*}
	where $C=\normX{\dot{x}(0)}^{2}+\normY{\dot{\lambda}(0)}^{2}+(\beta/2+\gamma^{2})\big(\normX{x(0)-\xi}^{2}+\normY{\lambda(0)-\eta}^{2}\big)$. Taking into account that $\norm{(\dot{x}(t),\dot{\lambda}(t))}^{2}\geq0$ and $\norm{(x(t),\lambda(t))-(\xi,\eta)}^{2}\geq0$, and subsequently passing to the limit as $t\to+\infty$ yields
	\begin{equation*}
		\int_{0}^{\infty}\normX{x(\tau)-\xi}^{2}+
		\normY{\lambda(\tau)-\eta}^{2}\,\mathrm{d}\tau<+\infty.
	\end{equation*}
	Since $t\mapsto\norm{(x(t),\lambda(t))-(\xi,\eta)}^{2}$ is non-increasing on $[0,+\infty)$, cf. Proposition \ref{sec2:pr:lagrangian}{\it (i)}, we have for any $t\geq0$,
	\begin{equation*}
		\int_{t/2}^{t}\normX{x(\tau)-\xi}^{2}+
		\normY{\lambda(\tau)-\eta}^{2}\,\mathrm{d}\tau\geq
		t\big(\normX{x(t)-\xi}^{2}+\normY{\lambda(t)-\eta}^{2}
		\big)/2.
	\end{equation*}
	Noticing that $\norm{(x,\lambda)-(\xi,\eta)}^{2}$ belongs to $\LL^{1}([0,+\infty);\rl)$, we classically deduce
	\begin{equation*}
		\lim_{t\to+\infty}t\big(\normX{x(t)-\xi}^{2}+
		\normY{\lambda(t)-\eta}^{2}\big)=0.
	\end{equation*}
	Finally, recall from \eqref{sec3:cauchyschwarz} that for any $t\geq0$,
	\begin{equation*}
		L(x(t),\eta)-L(\xi,\lambda(t))+\alpha\normX{x(t)-\xi}^{2}/2
		\leq\norm{(x(t),\lambda(t))-(\xi,\eta)}
		\norm{(\dot{x}(t),\dot{\lambda}(t))}.
	\end{equation*}
	Multiplying the above inequality by $t$ and using $\lim_{t\to+\infty}\sqrt{t}\norm{(\dot{x}(t),\dot{\lambda}(t))}=0$, cf. Theorem \ref{sec3:th:weakconv}, together with $\lim_{t\to+\infty}\sqrt{t}\norm{(x(t),\lambda(t))-(\xi,\eta)}=0$, we infer
	\begin{equation*}
		\lim_{t\to+\infty}t\big(L(x(t),\eta)-L(\xi,\lambda(t))+
		\alpha\normX{x(t)-\xi}^{2}/2\big)=0,
	\end{equation*}
	concluding the desired estimates.
\end{proof}

\section{Exponential decay rate estimates}\label{sec4}
In this section, we provide decay rate estimates of exponential type on the solutions of \eqref{sec2:arrowhurwicz} under the additional assumption that $f$ is twice continuously differentiable. In particular, we presuppose that
\begin{enumerate}[\hspace{6pt}({A}1)]
	\setcounter{enumi}{5}
	\item $f:X\to\rl$ satisfies condition ${\rm (C)}$, i.e.,
	\begin{equation*}
		2D_{f}(y,x)-\iprodX{\nabla^{2}f(x)(x-y)}{x-y}\geq0,\quad\forall x,y\in X;
	\end{equation*}
	\item $\nabla^{2}f(\,\cdot\,):X\to X$ is $\gamma$-bounded, i.e.,
	\begin{equation*}
		\exists\gamma>0\ \forall x,y\in X\quad
		\iprodX{\nabla^{2}f(x)y}{y}\leq\gamma\normX{y}^{2}.
	\end{equation*}
\end{enumerate}
Here, $D_{f}$ denotes again the Bregman distance associated with $f$, cf. Bregman \cite{LMB:67}, and $\nabla^{2}f$ refers to the Hessian of $f$. We remark that condition ${\rm (C)}$ is verified whenever $f$ is minorized by its second-order Taylor approximations.

\subsection{`Primal exponential estimates'}
Let us first establish exponential decay rate estimates on the solutions of \eqref{sec2:arrowhurwicz} in the case when $A:X\to Y$ is bounded from below, i.e.,
\begin{equation*}
	\exists\beta>0\ \forall x\in X\quad\normY{Ax}^{2}\geq\beta\normX{x}^{2}.
\end{equation*}
\begin{theorem}\label{sec4:th:expconv}
	Let $\nabla^{2}f(\,\cdot\,):X\to X$ be $\alpha$-elliptic and $\gamma$-bounded, and suppose that $A:X\to Y$ is bounded from below with constant $\beta$. Let $f:X\to\rl$ satisfy condition {\rm (C)} and set
	\begin{equation*}
		\rho=
		\begin{cases}
			\hspace{45pt}\alpha/2,&\text{if $\gamma^{2}\leq4\beta$,}\\
			\min\{\alpha,\gamma-\sqrt{\gamma^{2}-4\beta}\}/2,
			&\text{if $\gamma^{2}>4\beta$}.
		\end{cases}
	\end{equation*}
	Let $(x,\lambda):[0,+\infty)\to X\times Y$ be a solution of \eqref{sec2:arrowhurwicz}. Then, for any $(\xi,\eta)\in S\times M$, the following assertions hold:
	\begin{enumerate}[(i)]
		\item If $\rho^{2}-\gamma\rho+\beta>0$, then it holds that
		\begin{align*}
			L(x(t),\eta)-L(\xi,\lambda(t))&=
			\OO\big(\mathrm{e}^{-2\rho t}\big)\ \text{as $t\to+\infty$};\\
			\norm{(\dot{x}(t),\dot{\lambda}(t))}^{2}&=
			\OO\big(\mathrm{e}^{-2\rho t}\big)\ \text{as $t\to+\infty$};\\
			\normX{x(t)-\xi}^{2}&=
			\OO\big(\mathrm{e}^{-2\rho t}\big)\ \text{as $t\to+\infty$};
		\end{align*}
		\item If $\rho^{2}-\gamma\rho+\beta=0$, then it holds that
		\begin{align*}
			L(x(t),\eta)-L(\xi,\lambda(t))&=
			\OO\big(t^{2}\mathrm{e}^{-2\rho t}\big)\ \text{as $t\to+\infty$};\\
			\norm{(\dot{x}(t),\dot{\lambda}(t))}^{2}&=
			\OO\big(t^{2}\mathrm{e}^{-2\rho t}\big)\ \text{as $t\to+\infty$};\\
			\normX{x(t)-\xi}^{2}&=
			\OO\big(t^{2}\mathrm{e}^{-2\rho t}\big)\ \text{as $t\to+\infty$}.
		\end{align*}
	\end{enumerate}
\end{theorem}
\begin{proof}
	Let $(\xi,\eta)\in S\times M$ and let $\rho>0$ to be chosen. Using again the `Lagrangian identity' \eqref{sec1:lagidentity} together with \eqref{sec2:arrowhurwicz}, we have for any $t\geq0$,
	\begin{align*}
		\frac{\mathrm{d}}{\mathrm{d}t}\big(\normX{\dot{x}(t)}^{2}+
		\normY{\dot{\lambda}(t)}^{2}&\big)/2+
		\iprodX{\nabla^{2}f(x(t))\dot{x}(t)}{\dot{x}(t)}\\
		+\,\rho\frac{\mathrm{d}}{\mathrm{d}t}\big(L(\xi,\eta)-
		L(x(t),\lambda(t))\big)&+\rho\big(\normX{\dot{x}(t)}^{2}+
		\normY{\dot{\lambda}(t)}^{2}\big)=2\rho\normX{\dot{x}(t)}^{2}.
	\end{align*}
	Moreover, from equation \eqref{sec3:lagdifequality}, we obtain
	\begin{align*}
		\frac{\mathrm{d}}{\mathrm{d}t}
		\big(L(\xi,\eta)&-L(x(t),\lambda(t))\big)=
		\frac{\mathrm{d}}{\mathrm{d}t}\iprodX{x(t)-\xi}{\dot{x}(x)}\\
		&+\iprodX{\nabla^{2}f(x(t))\dot{x}(t)}{x(t)-\xi}.
	\end{align*}
	Combining the above expressions yields
	\begin{align*}
		\frac{\mathrm{d}}{\mathrm{d}t}\big(\normX{\dot{x}(t)+
		\rho(x(t)&-\xi)}^{2}+\rho^{2}\normX{x(t)-\xi}^{2}+
		\normY{\dot{\lambda}(t)}^{2}\big)/2\\
		+\,\iprodX{(\nabla^{2}f(x(t))&-2\rho\id)\dot{x}(t)}{\dot{x}(t)+
		\rho(x(t)-\xi)}\\
		+\,\rho&\big(\normX{\dot{x}(t)}^{2}+\normY{\dot{\lambda}(t)}^{2}\big)=0.
	\end{align*}
	Developing the term in the second line gives
	\begin{align*}
		\frac{\mathrm{d}}{\mathrm{d}t}\big(\normX{\dot{x}(t)+
		\rho(x(t)-\xi)}^{2}+\rho^{2}\normX{x(t)-\xi}^{2}-
		2\rho D_{f}(\xi,x(t))&+\normY{\dot{\lambda}(t)}^{2}\big)/2\\
		+\,\rho\big(\normX{\dot{x}(t)+\rho(x(t)-\xi)}^{2}+
		\rho^{2}\normX{x(t)-\xi}^{2}-2\rho D_{f}(\xi,x(t))&+
		\normY{\dot{\lambda}(t)}^{2}\big)\\
		+\,\iprodX{(\nabla^{2}f(x(t))-2\rho\id)(\dot{x}(t)+
		\rho(x(t)-\xi))}{\dot{x}(t)+\rho(x(t)&-\xi)}\\
		+\,\rho^{2}\big(2D_{f}(\xi,x(t))-\iprodX{\nabla^{2}f(x(t))(x(t)-
		\xi)}{x(t)-\xi}\big)&=0.
	\end{align*}
	Since $\nabla^{2}f(\,\cdot\,)$ is $\alpha$-elliptic and $f$ satisfies condition {\rm (C)}, we obtain
	\begin{align*}
		\frac{\mathrm{d}}{\mathrm{d}t}\big(\normX{\dot{x}(t)+
		\rho(x(t)-\xi)}^{2}&+\rho^{2}\normX{x(t)-\xi}^{2}-
		2\rho D_{f}(\xi,x(t))+\normY{\dot{\lambda}(t)}^{2}\big)/2\\
		+\,\rho\big(\normX{\dot{x}(t)+\rho(x(t)-\xi)}^{2}&+
		\rho^{2}\normX{x(t)-\xi}^{2}-2\rho D_{f}(\xi,x(t))+
		\normY{\dot{\lambda}(t)}^{2}\big)\\
		+\,(\alpha-2&\rho)\normX{\dot{x}(t)+\rho(x(t)-\xi)}^{2}\leq0.
	\end{align*}
	An immediate integration over $[0,t]$ shows that there exists $C\geq0$ such that
	\begin{equation}\label{sec4:primalexp}
		\begin{split}
			\normX{\dot{x}(t)&+\rho(x(t)-\xi)}^{2}+
			\rho^{2}\normX{x(t)-\xi}^{2}-
			2\rho D_{f}(\xi,x(t))+\normY{\dot{\lambda}(t)}^{2}\\
			+\,2&(\alpha-2\rho)\int_{0}^{t}\mathrm{e}^{-2\rho(t-\tau)}
			\normX{\dot{x}(\tau)+\rho(x(\tau)-\xi)}^{2}\,\mathrm{d}\tau
			\leq C\mathrm{e}^{-2\rho t}.
		\end{split}
	\end{equation}
	Using that $\nabla^{2}f(\,\cdot\,)$ is $\gamma$-bounded and that $A$ is bounded from below with constant $\beta$, we have both $D_{f}(\xi,x(t))\leq\gamma\normX{x(t)-\xi}^{2}/2$ and $\normY{A(x(t)-\xi)}^{2}\geq\beta\normX{x(t)-\xi}^{2}$. In view of \eqref{sec2:arrowhurwicz} and $A\xi-b=0_{Y}$, we infer
	\begin{equation}\label{sec4:expinequ}
		\begin{split}
			\normX{\dot{x}(t)+\rho(x(t)-\xi)}^{2}+(\rho^{2}-\gamma\rho&+\beta)
			\normX{x(t)-\xi}^{2}\\
			+\,2(\alpha-2\rho)\int_{0}^{t}\mathrm{e}^{-2\rho(t-\tau)}
			\normX{\dot{x}(\tau)+\rho(x(\tau)&-\xi)}^{2}\,\mathrm{d}\tau
			\leq C\mathrm{e}^{-2\rho t}.
		\end{split}
	\end{equation}

	Let us now determine the largest value for $\rho\in(0,\alpha/2]$ such that $\rho^{2}-\gamma\rho+\beta\geq0$. Clearly, if $\gamma^{2}\leq4\beta$, then $\rho^{2}-\gamma\rho+\beta\geq0$ holds for any $\rho>0$. On the other hand, if $\gamma^{2}>4\beta$, then $\rho^{2}-\gamma\rho+\beta\geq0$ is attained whenever $\rho\leq\gamma/2-\sqrt{\gamma^{2}-4\beta}/2$. Consequently, we may take
	\begin{equation*}
		\rho=
		\begin{cases}
			\hspace{45pt}\alpha/2,&\text{if $\gamma^{2}\leq4\beta$,}\\
			\min\{\alpha,\gamma-\sqrt{\gamma^{2}-4\beta}\}/2,
			&\text{if $\gamma^{2}>4\beta$}.
		\end{cases}
	\end{equation*}
	We have either one of the following cases:

	{\it (i)} Suppose that $\rho^{2}-\gamma\rho+\beta>0$. In this case, we deduce from \eqref{sec4:expinequ} that
	\begin{align*}
		\mathrm{e}^{2\rho t}\normX{\dot{x}(t)+\rho(x(t)-\xi)}^{2}&\leq C;\\
		\mathrm{e}^{2\rho t}\normX{x(t)-\xi}^{2}&\leq
		\frac{C}{\rho^{2}-\gamma\rho+\beta}.
	\end{align*}
	Passing to the upper limit as $t\to+\infty$ yields
	\begin{align*}
		\limsup_{t\to+\infty}\,\mathrm{e}^{2\rho t}
		\normX{\dot{x}(t)+\rho(x(t)-\xi)}^{2}&<+\infty;\\
		\limsup_{t\to+\infty}\,\mathrm{e}^{2\rho t}
		\normX{x(t)-\xi}^{2}&<+\infty.
	\end{align*}
	Moreover, in view of the basic inequality
	\begin{equation*}
		\normX{\dot{x}(t)}^{2}\leq2\normX{\dot{x}(t)+\rho(x(t)-\xi)}^{2}+
		2\rho^{2}\normX{x(t)-\xi}^{2}
	\end{equation*}
	and the fact that
	\begin{equation*}
		\normY{\dot{\lambda}(t)}^{2}=\normY{A(x(t)-\xi)}^{2}\leq
		\norm{A}^{2}\normX{x(t)-\xi}^{2},
	\end{equation*}
	we obtain
	\begin{equation*}
		\limsup_{t\to+\infty}\,\mathrm{e}^{2\rho t}\big(\normX{\dot{x}(t)}^{2}+
		\normY{\dot{\lambda}(t)}^{2}\big)<+\infty.
	\end{equation*}
	The remaining estimate is now readily deduced as in Corollary \ref{sec3:co:refinedgap}.

	{\it (ii)} Suppose now that $\rho^{2}-\gamma\rho+\beta=0$. In this case, we observe from \eqref{sec4:expinequ} that
	\begin{equation*}
		\mathrm{e}^{2\rho t}\normX{\dot{x}(t)+\rho(x(t)-\xi)}^{2}\leq C.
	\end{equation*}
	Passing to the upper limit as $t\to+\infty$ entails
	\begin{equation*}
		\limsup_{t\to+\infty}\,\mathrm{e}^{2\rho t}
		\normX{\dot{x}(t)+\rho(x(t)-\xi)}^{2}<+\infty.
	\end{equation*}
	Moreover, given the fact that
	\begin{equation*}
		\mathrm{e}^{\rho t}\normX{x(t)-\xi}\leq\normX{x(0)-\xi}+
		\int_{0}^{t}\mathrm{e}^{\rho\tau}\normX{\dot{x}(\tau)+\rho(x(\tau)-\xi)}\,
		\mathrm{d}\tau,
	\end{equation*}
	we deduce
	\begin{equation*}
		\mathrm{e}^{\rho t}\normX{x(t)-\xi}\leq\sqrt{C}t+\normX{x(0)-\xi}.
	\end{equation*}
	Taking the square and multiplying the resulting inequality by $t^{-2}$ yields
	\begin{equation*}
		t^{-2}\mathrm{e}^{2\rho t}\normX{x(t)-\xi}^{2}\leq
		C+2\sqrt{C}\normX{x(0)-\xi}t^{-1}+\normX{x(0)-\xi}^{2}t^{-2}.
	\end{equation*}
	This majorization being valid for any $t>0$, we conclude
	\begin{equation*}
		\limsup_{t\to+\infty}\,t^{-2}\mathrm{e}^{2\rho t}
		\normX{x(t)-\xi}^{2}<+\infty.
	\end{equation*}
	The remaining estimates now follow at once.
\end{proof}

The previous result complements the exponential decay rate estimates obtained by Polyak \cite{BTP:70} based on spectral arguments. We further note that the above decay rate es- timates are comparable to the spectral bounds known for `saddle matrices'; cf. the sur- vey paper by Benzi et al. \cite[Section 3.4]{MB-GHG-JL:05} and references therein.

Assuming, moreover, that $\nabla^{2}f(\,\cdot\,)=\alpha\id$, we have the following refined exponential decay rate estimates.
\begin{corollary}\label{sec4:co:expconv}
	Let $\nabla^{2}f(\,\cdot\,)=\alpha\id$ and suppose that $A:X\to Y$ is bounded from below with constant $\beta$. Let $(x,\lambda):[0,+\infty)\to X\times Y$ be a solution of \eqref{sec2:arrowhurwicz}. Then, for any $(\xi,\eta)\in S\times M$, the following assertions hold:
	\begin{enumerate}[(i)]
		\item If $\alpha^{2}<4\beta$, then it holds that
		\begin{align*}
			L(x(t),\eta)-L(\xi,\lambda(t))&=
			\OO\big(\mathrm{e}^{-\alpha t}\big)\ \text{as $t\to+\infty$};\\
			\norm{(\dot{x}(t),\dot{\lambda}(t))}^{2}&=
			\OO\big(\mathrm{e}^{-\alpha t}\big)\ \text{as $t\to+\infty$};\\
			\normX{x(t)-\xi}^{2}&=
			\OO\big(\mathrm{e}^{-\alpha t}\big)\ \text{as $t\to+\infty$};
		\end{align*}
		\item If $\alpha^{2}=4\beta$, then it holds that
		\begin{align*}
			L(x(t),\eta)-L(\xi,\lambda(t))&=
			\OO\big(t^{2}\mathrm{e}^{-\alpha t}\big)\ \text{as $t\to+\infty$};\\
			\norm{(\dot{x}(t),\dot{\lambda}(t))}^{2}&=
			\OO\big(t^{2}\mathrm{e}^{-\alpha t}\big)\ \text{as $t\to+\infty$};\\
			\normX{x(t)-\xi}^{2}&=
			\OO\big(t^{2}\mathrm{e}^{-\alpha t}\big)\ \text{as $t\to+\infty$};
		\end{align*}
		\item If $\alpha^{2}>4\beta$, then it holds that
		\begin{align*}
			L(x(t),\eta)-L(\xi,\lambda(t))&=
			\OO\big(\mathrm{e}^{-(\alpha-\delta)t}\big)\ \text{as $t\to+\infty$};\\
			\norm{(\dot{x}(t),\dot{\lambda}(t))}^{2}&=
			\OO\big(\mathrm{e}^{-(\alpha-\delta)t}\big)\ \text{as $t\to+\infty$};\\
			\normX{x(t)-\xi}^{2}&=
			\OO\big(\mathrm{e}^{-(\alpha-\delta)t}\big)\ \text{as $t\to+\infty$},
		\end{align*}
		where $\delta=\sqrt{\alpha^{2}-4\beta}$.
	\end{enumerate}
\end{corollary}
\begin{proof}
	{\it (i)--(ii)} This is an immediate consequence of Theorem \ref{sec4:th:expconv}{\it (i)--(ii)}.

	{\it (iii)} Suppose that $\alpha^{2}>4\beta$ and let $\rho=(\alpha+\delta)/2$, where $\delta=\sqrt{\alpha^{2}-4\beta}$, so that $\rho^{2}-\alpha\rho+\beta=0$. From \eqref{sec4:expinequ} and the fact that $\nabla^{2}f(\,\cdot\,)=\alpha\id$, we observe that there exists $C\geq0$ such that for any $t\geq0$,
	\begin{equation*}
		\mathrm{e}^{(\alpha+\delta)t}\normX{\dot{x}(t)+\rho(x(t)-\xi)}^{2}\leq
		C+2\delta\int_{0}^{t}\mathrm{e}^{(\alpha+\delta)\tau}
		\normX{\dot{x}(\tau)+\rho(x(\tau)-\xi)}^{2}\,\mathrm{d}\tau.
	\end{equation*}
	Applying Gronwall's inequality yields
	\begin{equation*}
		\mathrm{e}^{(\alpha+\delta)t}\normX{\dot{x}(t)+\rho(x(t)-\xi)}^{2}\leq
		C\mathrm{e}^{2\delta t}.
	\end{equation*}
	Using this inequality together with the fact that
	\begin{equation*}
		\mathrm{e}^{(\alpha+\delta)t/2}\normX{x(t)-\xi}\leq
		\normX{x(0)-\xi}+\int_{0}^{t}\mathrm{e}^{(\alpha+\delta)\tau/2}
		\normX{\dot{x}(\tau)+\rho(x(\tau)-\xi)}\,\mathrm{d}\tau,
	\end{equation*}
	we obtain
	\begin{equation*}
		\mathrm{e}^{(\alpha+\delta)t/2}\normX{x(t)-\xi}\leq
		\frac{\sqrt{C}}{\delta}\mathrm{e}^{\delta t}+\normX{x(0)-\xi}-
		\frac{\sqrt{C}}{\delta}.
	\end{equation*}
	Consequently,
	\begin{align*}
		\mathrm{e}^{(\alpha-\delta)t}\normX{x(t)-\xi}^{2}&\leq
		\frac{C}{\delta^{2}}+2\frac{\sqrt{C}}{\delta}\Big(\normX{x(0)-\xi}-
		\frac{\sqrt{C}}{\delta}\Big)\mathrm{e}^{-\delta t}\\
		&\quad+\Big(\normX{x(0)-\xi}-\frac{\sqrt{C}}{\delta}\Big)^{2}
		\mathrm{e}^{-2\delta t}.
	\end{align*}
	Passing to the upper limit as $t\to+\infty$ yields the desired estimate. The remaining as- sertions are now easily obtained.
\end{proof}
\begin{remark}\label{sec4:rm:cdhosc}
	The previous result essentially recovers the optimal decay rate estimates known for the classical damped harmonic oscillator. Indeed, in the case when $\nabla^{2}f(\,\cdot\,)=\alpha\id$, we observe in view of an immediate differentiation that the solutions of \eqref{sec2:arrowhurwicz} fur- ther obey the second-order dynamics
	\begin{equation*}
		\ddot{x}+\alpha\dot{x}+\nabla\normY{A(x-\bar{x})}^{2}/2=0_{X},
	\end{equation*}
	where $\bar{x}$ denotes the unique element in $S$. The above second-order differential system was first introduced, from a more general optimization perspective, by Polyak \cite{BTP:64} and is known to inherit remarkable minimizing properties; see, e.g., Alvarez \cite{FA:00} and Attouch et al. \cite{HA-XG-PR:00} for a general exposition.
\end{remark}

\subsection{`Dual exponential estimates'}
Let us now complement the previous discussion with decay rate estimates on the solutions of \eqref{sec2:arrowhurwicz} under the assumption that $A^{\ast}:Y\to X$ is bounded from below, i.e.,
\begin{equation*}
	\exists\beta>0\ \forall y\in Y\quad\normX{A^{\ast}y}^{2}\geq\beta\normY{y}^{2}.
\end{equation*}
\begin{proposition}\label{sec4:pr:dualexp}
	Let $\nabla^{2}f(\,\cdot\,)=\alpha\id$ and suppose that $A^{\ast}:Y\to X$ is bounded from below with constant $\beta$. Let $(x,\lambda):[0,+\infty)\to X\times Y$ be a solution of \eqref{sec2:arrowhurwicz}. Then, for $(\xi,\eta)\in S\times M$, the following assertions hold:
	\begin{enumerate}[(i)]
		\item If $\alpha^{2}<4\beta$, then it holds that
		\begin{align*}
			\normY{\lambda(t)-\eta}^{2}&=
			\OO\big(\mathrm{e}^{-\alpha t}\big)\ \text{as $t\to+\infty$};\\
			\normY{\dot{\lambda}(t)}^{2}&=
			\OO\big(\mathrm{e}^{-\alpha t}\big)\ \text{as $t\to+\infty$};
		\end{align*}
		\item If $\alpha^{2}=4\beta$, then it holds that
		\begin{align*}
			\normY{\lambda(t)-\eta}^{2}&=
			\OO\big(t^{2}\mathrm{e}^{-\alpha t}\big)\ \text{as $t\to+\infty$};\\
			\normY{\dot{\lambda}(t)}^{2}&=
			\OO\big(t^{2}\mathrm{e}^{-\alpha t}\big)\ \text{as $t\to+\infty$};
		\end{align*}
		\item If $\alpha^{2}>4\beta$, then it holds that
		\begin{align*}
			\normY{\lambda(t)-\eta}^{2}&=
			\OO\big(\mathrm{e}^{-(\alpha-\delta)t}\big)\ \text{as $t\to+\infty$};\\
			\normY{\dot{\lambda}(t)}^{2}&=
			\OO\big(\mathrm{e}^{-(\alpha-\delta)t}\big)\ \text{as $t\to+\infty$},
		\end{align*}
		where $\delta=\sqrt{\alpha^{2}-4\beta}$.
	\end{enumerate}
\end{proposition}
\begin{proof}
	Let $(\xi,\eta)$ be the unique element in $S\times M$ and let $\rho>0$. Using similar derivations as in Theorem \ref{sec4:th:expconv}, we have for any $t\geq0$
	\begin{align*}
		\frac{\mathrm{d}}{\mathrm{d}t}\big(\normY{\dot{\lambda}(t)
		+\rho(\lambda(t)-\eta)}^{2}&+\rho^{2}\normY{\lambda(t)-\eta}^{2}+
		\normX{A^{\ast}(\lambda(t)-\eta)}^{2}\big)/2\\
		+\,\rho\big(\normY{\dot{\lambda}(t)+\rho(\lambda(t)-\eta)}^{2}&+
		\rho^{2}\normY{\lambda(t)-\eta}^{2}+
		\normX{A^{\ast}(\lambda(t)-\eta)}^{2}\big)\\
		+\,\iprodY{A(\nabla f(x(t))\,-&\,\nabla f(\xi))}{\dot{\lambda}(t)+
		\rho(\lambda(t)-\eta)}\\
		-\,2\rho\normY{\dot{\lambda}(t)&+\rho(\lambda(t)-\eta)}^{2}=0.
	\end{align*}
	From $\nabla^{2}f(\,\cdot\,)=\alpha\id$ together with \eqref{sec2:arrowhurwicz} and $A\xi-b=0_{Y}$, we obtain
	\begin{align*}
		\frac{\mathrm{d}}{\mathrm{d}t}\big(\normY{\dot{\lambda}(t)+
		\rho(\lambda(t)-\eta)}^{2}+(\rho^{2}-\alpha\rho)
		\normY{\lambda(t)-\eta}^{2}&+
		\normX{A^{\ast}(\lambda(t)-\eta)}^{2}\big)/2\\
		+\,\rho\big(\normY{\dot{\lambda}(t)+\rho(\lambda(t)-\eta)}^{2}+
		(\rho^{2}-\alpha\rho)\normY{\lambda(t)-\eta}^{2}&+
		\normX{A^{\ast}(\lambda(t)-\eta)}^{2}\big)\\
		+\,(\alpha-2\rho)\normY{\dot{\lambda}(t)+
		\rho(\lambda(t)-\eta)}^{2}&=0.
	\end{align*}
	An immediate integration over $[0,t]$ shows that there exists $C\geq0$ such that
	\begin{equation}\label{sec4:dualexp}
		\begin{split}
			\normY{\dot{\lambda}(t)+\rho(\lambda(t)-\eta)}^{2}+
			(\rho^{2}-\alpha\rho)\normY{\lambda(t)&-\eta}^{2}+
			\normX{A^{\ast}(\lambda(t)-\eta)}^{2}\\
			+\,2(\alpha-2\rho)\int_{0}^{t}\normY{\dot{\lambda}(\tau)+
			\rho(\lambda(\tau)&-\eta)}^{2}\,\mathrm{d}\tau=
			C\mathrm{e}^{-2\rho t}.
		\end{split}
	\end{equation}
	Since $A^{\ast}$ is bounded from below with constant $\beta$, we infer
	\begin{align*}
		\normY{\dot{\lambda}(t)+\rho(\lambda(t)-\eta)}^{2}+
		(\rho^{2}-\alpha\rho+\beta)\normY{\lambda(t)-\eta}^{2}\\
		+\,2(\alpha-2\rho)\int_{0}^{t}\normY{\dot{\lambda}(\tau)+
		\rho(\lambda(\tau)-\eta)}^{2}\,\mathrm{d}\tau\leq
		C\mathrm{e}^{-2\rho t}.
	\end{align*}
	The desired estimates are now readily deduced.
\end{proof}
\begin{remark}
	The above result again retrieves the well-known decay rate estimates for the classical damped harmonic oscillator. As in the previous case (and, in fact, dual to our observation in Remark \ref{sec4:rm:cdhosc}), we note that whenever $\nabla^{2}f(\,\cdot\,)=\alpha\id$, the solutions of \eqref{sec2:arrowhurwicz} further obey the second-order dynamics
	\begin{equation*}
		\ddot{\lambda}+\alpha\dot{\lambda}+\nabla
		\normX{A^{\ast}(\lambda-\bar{\lambda})}^{2}/2=0_{Y}
	\end{equation*}
	with $\bar{\lambda}$ denoting the unique element in $M$. We leave the details to the reader.
\end{remark}

In order to obtain asymptotic estimates on the primal-dual gap function in the case when $A^{\ast}$ is bounded from below, we utilize the following relation between the primal and dual variables.
\begin{lemma}\label{sec4:lm:combbound}
	Let $\nabla^{2}f(\,\cdot\,)=\alpha\id$, let $A^{\ast}:Y\to X$ be bounded from below with constant $\beta$, and let $(x,\lambda):[0,+\infty)\to X\times Y$ be a solution of \eqref{sec2:arrowhurwicz}. Then, for $(\xi,\eta)\in S\times M$, there exists $C\geq0$ such that for any $t\geq0$,
	\begin{equation*}
		\beta\norm{(x(t),\lambda(t))-(\xi,\eta)}^{2}-
		\normY{A(x(t)-\xi)}^{2}-\normX{A^{\ast}(\lambda(t)-\eta)}^{2}
		\leq C\mathrm{e}^{-2\alpha t}.
	\end{equation*}
\end{lemma}
\begin{proof}
	Let $(\xi,\eta)$ be the unique element in $S\times M$. Using \eqref{sec2:arrowhurwicz} together with the fact that $T(\xi,\eta)=(0_{X},0_{Y})$, we have for any $t\geq0$,
	\begin{align*}
		\frac{\mathrm{d}}{\mathrm{d}t}\big(\beta\normX{x(t)-\xi}^{2}+
		\beta\normY{\lambda(t)-\eta}^{2}-\normY{A(x(t)-\xi)}^{2}&-
		\normX{A^{\ast}(\lambda(t)-\eta)}^{2}\big)/2\\
		+\,\beta\iprodX{\nabla f(x(t))-\nabla f(\xi)}{x(t)-\xi}-
		\iprodX{\nabla f(x(t))-\nabla f(&\xi)}{A^{\ast}A(x(t)-\xi)}=0.
	\end{align*}
	In view of $\nabla^{2}f(\,\cdot\,)=\alpha\id$, the above equality reads
	\begin{align*}
		\frac{\mathrm{d}}{\mathrm{d}t}\big(\beta\normX{x(t)-\xi}^{2}+
		\beta\normY{\lambda(t)-\eta}^{2}-\normY{A(x(t)-\xi)}^{2}&-
		\normX{A^{\ast}(\lambda(t)-\eta)}^{2}\big)/2\\
		+\,\alpha\big(\beta\normX{x(t)-\xi}^{2}+\beta
		\normY{\lambda(t)-\eta}^{2}-\normY{A(x(t)-\xi)}^{2}&-
		\normX{A^{\ast}(\lambda(t)-\eta)}^{2}\big)\\
		+\,\alpha\big(\normX{A^{\ast}(\lambda(t)-\eta)}^{2}-\beta
		\normY{\lambda(t)-\eta}^{2}\big)&=0.
	\end{align*}
	Since $A^{\ast}$ is bounded from below with constant $\beta$, we obtain
	\begin{align*}
		\frac{\mathrm{d}}{\mathrm{d}t}\big(\beta\normX{x(t)-\xi}^{2}+
		\beta\normY{\lambda(t)-\eta}^{2}-\normY{A(x(t)-\xi)}^{2}&-
		\normX{A^{\ast}(\lambda(t)-\eta)}^{2}\big)/2\\
		+\,\alpha\big(\beta\normX{x(t)-\xi}^{2}+\beta
		\normY{\lambda(t)-\eta}^{2}-\normY{A(x(t)-\xi)}^{2}&-
		\normX{A^{\ast}(\lambda(t)-\eta)}^{2}\big)\leq0.
	\end{align*}
	An immediate integration over $[0,t]$ then shows that there exists $C\geq0$ such that
	\begin{equation*}
		\beta\normX{x(t)-\xi}^{2}+\beta
		\normY{\lambda(t)-\eta}^{2}-\normY{A(x(t)-\xi)}^{2}-
		\normX{A^{\ast}(\lambda(t)-\eta)}^{2}\leq C\mathrm{e}^{-2\alpha t}.
		\qedhere
	\end{equation*}
\end{proof}

Combining the above results finally gives the following asymptotic estimates.
\begin{corollary}\label{sec4:co:dualexpconv}
	Let $\nabla^{2}f(\,\cdot\,)=\alpha\id$ and suppose that $A^{\ast}:Y\to X$ is bounded from below with constant $\beta$. Let $(x,\lambda):[0,+\infty)\to X\times Y$ be a solution of \eqref{sec2:arrowhurwicz}. Then, for $(\xi,\eta)\in S\times M$, the following assertions hold:
	\begin{enumerate}[(i)]
		\item If $\alpha^{2}<4\beta$, then it holds that
		\begin{align*}
			L(x(t),\eta)-L(\xi,\lambda(t))&=
			\OO\big(\mathrm{e}^{-\alpha t}\big)\ \text{as $t\to+\infty$};\\
			\norm{(\dot{x}(t),\dot{\lambda}(t))}^{2}&=
			\OO\big(\mathrm{e}^{-\alpha t}\big)\ \text{as $t\to+\infty$};\\
			\norm{(x(t),\lambda(t))-(\xi,\eta)}^{2}&=
			\OO\big(\mathrm{e}^{-\alpha t}\big)\ \text{as $t\to+\infty$};
		\end{align*}
		\item If $\alpha^{2}=4\beta$, then it holds that
		\begin{align*}
			L(x(t),\eta)-L(\xi,\lambda(t))&=
			\OO\big(t^{2}\mathrm{e}^{-\alpha t}\big)\ \text{as $t\to+\infty$};\\
			\norm{(\dot{x}(t),\dot{\lambda}(t))}^{2}&=
			\OO\big(t^{2}\mathrm{e}^{-\alpha t}\big)\ \text{as $t\to+\infty$};\\
			\norm{(x(t),\lambda(t))-(\xi,\eta)}^{2}&=
			\OO\big(t^{2}\mathrm{e}^{-\alpha t}\big)\ \text{as $t\to+\infty$};
		\end{align*}
		\item If $\alpha^{2}>4\beta$, then it holds that
		\begin{align*}
			L(x(t),\eta)-L(\xi,\lambda(t))&=
			\OO\big(\mathrm{e}^{-(\alpha-\delta)t}\big)\ \text{as $t\to+\infty$};\\
			\norm{(\dot{x}(t),\dot{\lambda}(t))}^{2}&=
			\OO\big(\mathrm{e}^{-(\alpha-\delta)t}\big)\ \text{as $t\to+\infty$};\\
			\norm{(x(t),\lambda(t))-(\xi,\eta)}^{2}&=
			\OO\big(\mathrm{e}^{-(\alpha-\delta)t}\big)\ \text{as $t\to+\infty$},
		\end{align*}
		where $\delta=\sqrt{\alpha^{2}-4\beta}$.
	\end{enumerate}
\end{corollary}
\begin{proof}
	Let $(\xi,\eta)$ be the unique element in $S\times M$ and let $\rho>0$. Since $\nabla^{2}f(\,\cdot\,)=\alpha\id$, by combining \eqref{sec4:primalexp} with \eqref{sec4:dualexp} and subsequently using \eqref{sec2:arrowhurwicz} together with the fact that $A\xi-b=0_{Y}$, there exists $C\geq0$ such that for any $t\geq0$,
	\begin{align*}
		\normX{\dot{x}(t)+\rho(x(t)-\xi)}^{2}+\normY{\dot{\lambda}(t)+
		\rho(\lambda(t)&-\eta)}^{2}+\normY{A(x(t)-\xi)}^{2}\\
		+\,(\rho^{2}-\alpha\rho)\big(\normX{x(t)-\xi}^{2}+
		\normY{\lambda(t)-\eta&}^{2}\big)+
		\normX{A^{\ast}(\lambda(t)-\eta)}^{2}\\
		+\,2(\alpha-2\rho)\int_{0}^{t}\mathrm{e}^{-2\rho(t-\tau)}\kappa(\tau)&\,
		\mathrm{d}\tau\leq C\mathrm{e}^{-2\rho t},
	\end{align*}
	where $\kappa(\tau)=\normX{\dot{x}(\tau)+\rho(x(\tau)-\xi)}^{2}+\normY{\dot{\lambda}(\tau)+\rho(\lambda(\tau)-\eta)}^{2}$. Since $A^{\ast}$ is bounded from below with constant $\beta$, we observe from Lemma \ref{sec4:lm:combbound} that there exists $\tilde{C}\geq0$ such that
	\begin{equation*}
		\beta\big(\normX{x(t)-\xi}^{2}+
		\normY{\lambda(t)-\eta}^{2}\big)-\normY{A(x(t)-\xi)}^{2}-
		\normX{A^{\ast}(\lambda(t)-\eta)}^{2}\leq\tilde{C}\mathrm{e}^{-2\alpha t}.
	\end{equation*}
	In view of the above inequalities, we obtain
	\begin{align*}
		\normX{\dot{x}(t)+\rho(x(t)-\xi)}^{2}+\normY{\dot{\lambda}(t)&+
			\rho(\lambda(t)-\eta)}^{2}\\
		+\,(\rho^{2}-\alpha\rho+\beta)\big(\normX{x(t)-\xi}^{2}&+
		\normY{\lambda(t)-\eta}^{2}\big)\\
		+\,2(\alpha-2\rho)\int_{0}^{t}\mathrm{e}^{-2\rho(t-\tau)}\kappa(\tau)\,
		\mathrm{d}\tau&\leq C\mathrm{e}^{-2\rho t}+\tilde{C}\mathrm{e}^{-2\alpha t}.
	\end{align*}
	The desired estimates are now easily derived.
\end{proof}

\section{Structured convex minimization}\label{sec5}
In this section, we aim to extend some of our previous results on the Arrow--Hurwicz differential system \eqref{sec1:arrowhurwicz} to the more general case of solving structured convex minimization problems. Let $X$, $Y$ and $Z$ be real Hilbert spaces, and let $X\times Y\times Z$ be endowed with the product structure $\iprod{\,\cdot\,}{\,\cdot\,}=\iprodX{\,\cdot\,}{\,\cdot\,}+\iprodY{\,\cdot\,}{\,\cdot\,}+\iprodZ{\,\cdot\,}{\,\cdot\,}$ and associated norm $\norm{\,\cdot\,}$. Consider the structured minimization problem
\begin{equation}
\renewcommand{\theequation}{SP}\tag{\theequation}\label{sec5:strucproblem}
	\inf\,\{f(x)+g(y)\mid Ax+By-c=0_{Z}\}
\end{equation}
and suppose that the following assumptions are verified:
\begin{enumerate}[\hspace{6pt}({A}1)$^{\prime}$]
	\item $f:X\to\rl$ and $g:Y\to\rl$ are convex and continuously differentiable;
	\item $\nabla f:X\to X$ and $\nabla g:Y\to Y$ are Lipschitz continuous on
	bounded sets;
	\item $A:X\to Z$ and $B:Y\to Z$ are linear and continuous, and $c\in Z$.
\end{enumerate}

We associate with \eqref{sec5:strucproblem} the Lagrangian
\begin{align*}
	L:X\times Y\times Z&\longrightarrow\rl\\
	(x,y,\lambda)&\longmapsto f(x)+g(y)+\iprodZ{\lambda}{Ax+By-c},
\end{align*}
which, given the above assumptions, is a convex function with respect to $(x,y)\in X\times Y$ and a concave (in fact, affine) function with respect to $\lambda\in Z$. We denote by $S\times U\times M\subset X\times Y\times Z$ the (possibly empty) set of saddle points of the Lagrangian $L$.

Consider now the generalized Arrow--Hurwicz differential system
\begin{equation}
\renewcommand{\theequation}{GAH}\tag{\theequation}\label{sec5:genarrowhurwicz}
	\begin{cases}
		\dot{x}+\nabla f(x)+A^{\ast}\lambda=0_{X}\\
		\hspace{1.48pt}\dot{y}+\nabla g(y)+B^{\ast}\lambda=0_{Y}\\
		\hspace{2.25pt}\dot{\lambda}+c-Ax-By=0_{Z}
	\end{cases}
\end{equation}
with initial data $(x_{0},y_{0},\lambda_{0})\in X\times Y\times Z$ and observe that \eqref{sec5:genarrowhurwicz} admits, for any ini- tial data, a unique (classical) solution $(x,y,\lambda):[0,+\infty)\to X\times Y\times Z$. Recall that the zeros of the maximally monotone operator
\begin{align*}
	T:X\times Y\times Z&\longrightarrow X\times Y\times Z\\
	(x,y,\lambda)&\longmapsto(\nabla f(x)+A^{\ast}\lambda,
	\nabla g(y)+B^{\ast}\lambda,c-Ax-By)
\end{align*}
are nothing but the saddle points of the Lagrangian $L$.

The following discussion suggests that our results on the \eqref{sec2:arrowhurwicz} differential system directly convey to the \eqref{sec5:genarrowhurwicz} evolution system for solving the structured convex minimization problem \eqref{sec5:strucproblem}.
\begin{theorem}
	Let $S\times U\times M$ be non-empty and let $(x,y,\lambda):[0,+\infty)\to X\times Y\times Z$ be a solution of \eqref{sec5:genarrowhurwicz}. Then, for any $(\xi,\psi,\eta)\in S\times U\times M$,
	\begin{enumerate}[(i)]
		\item $\lim_{t\to+\infty}\norm{(x(t),y(t),\lambda(t))-(\xi,\psi,\eta)}$
		exists;
		\item $\lim_{t\to+\infty}\norm{(\dot{x}(t),\dot{y}(t),\dot{\lambda}(t))}$
		exists;
		\item it holds that
		\begin{equation*}
			\int_{0}^{\infty}L(x(\tau),y(\tau),\eta)-L(\xi,\psi,\lambda(\tau))\,
			\mathrm{d}\tau<+\infty.
		\end{equation*}
	\end{enumerate}
\end{theorem}

Let the Ces{\`a}ro average $(\sigma,\tau,\omega):(0,+\infty)\to X\times Y\times Z$ of a solution of \eqref{sec5:genarrowhurwicz} be defined by
\begin{equation*}
	(\sigma(t),\tau(t),\omega(t))=\frac{1}{t}\int_{0}^{t}
	(x(\tau),y(\tau),\lambda(\tau))\,\mathrm{d}\tau.
\end{equation*}
We have the following asymptotic estimate on the primal-dual gap function.
\begin{proposition}
	Let $S\times U\times M$ be non-empty and let $(\sigma,\tau,\omega):(0,+\infty)\to X\times Y\times Z$ be the Ces{\`a}ro average of a solution of \eqref{sec5:genarrowhurwicz}. Then, for any $(\xi,\psi,\eta)\in S\times U\times M$, it holds that
	\begin{equation*}
		L(\sigma(t),\tau(t),\eta)-L(\xi,\psi,\omega(t))=\OO\Big(\frac{1}{t}\Big)\
		\text{as $t\to+\infty$}.
	\end{equation*}
	Moreover, there exists $(\bar{x},\bar{y},\bar{\lambda})\in S\times U\times M$ such that $(\sigma(t),\tau(t),\omega(t))\rightharpoonup(\bar{x},\bar{y},\bar{\lambda})$ weakly in $X\times Y\times Z$ as $t\to+\infty$.
\end{proposition}

Let us now investigate the asymptotic properties of the solutions of \eqref{sec5:genarrowhurwicz} under the more stringent assumption that
\begin{align*}
	f+g:X\times Y&\longrightarrow\rl\\
	(x,y)&\longmapsto f(x)+g(y)
\end{align*}
is $\alpha$-strongly convex (or, equivalently, $\nabla(f+g):X\times Y\to X\times Y$ is $\alpha$-strongly mo- notone). In this case, the following asymptotic properties are verified.
\begin{theorem}\label{sec5:th:asympprop}
	Let $\nabla(f+g):X\times Y\to X\times Y$ be $\alpha$-strongly monotone, let $S\times U\times M$ be non-empty, and let $(x,y,\lambda):[0,+\infty)\to X\times Y\times Z$ be a solution of \eqref{sec5:genarrowhurwicz} with initial data $(x_{0},y_{0},\lambda_{0})\in X\times Y\times Z$. Then, for any $(\xi,\psi,\eta)\in S\times U\times M$, it holds that
	\begin{align*}
		L(x(t),y(t),\eta)-L(\xi,\psi,\lambda(t))&=
		o\Big(\frac{1}{\sqrt{t}}\Big)\ \text{as $t\to+\infty$};\\
		\norm{(\dot{x}(t),\dot{y}(t),\dot{\lambda}(t))}&=
		o\Big(\frac{1}{\sqrt{t}}\Big)\ \text{as $t\to+\infty$}.
	\end{align*}
	Moreover, $(x(t),y(t),\lambda(t))$ converges weakly, as $t\to+\infty$, to $\proj_{S\times U\times M}(x_{0},y_{0},\lambda_{0})\in S\times U\times M$.
\end{theorem}

Assuming, moreover, that $(A\ B)^{\ast}:Z\to X\times Y$ is bounded from below, i.e.,
\begin{equation*}
	\exists\beta>0\ \forall z\in Z\quad
	\norm{(A\ B)^{\ast}z}_{X\times Y}^{2}\geq\beta\normZ{z}^{2},
\end{equation*}
we have the following refined asymptotic estimates.
\begin{corollary}
	Under the hypotheses of Theorem \ref{sec5:th:asympprop}, let $(A\ B)^{\ast}:Z\to X\times Y$ be bounded from below. Then, for $(\xi,\psi,\eta)\in S\times U\times M$, it holds that
	\begin{align*}
		L(x(t),y(t),\eta)-L(\xi,\psi,\lambda(t))&=
		o\Big(\frac{1}{t}\Big)\ \text{as $t\to+\infty$};\\
		\norm{(x(t),y(t),\lambda(t))-(\xi,\psi,\eta)}&=
		o\Big(\frac{1}{\sqrt{t}}\Big)\ \text{as $t\to+\infty$}.
	\end{align*}
	Consequently, $(x(t),y(t),\lambda(t))$ converges strongly, as $t\to+\infty$, to the unique element in $S\times U\times M$.
\end{corollary}
\begin{remark}
	The structured convex minimization problem \eqref{sec5:strucproblem} has recently been~ap- proached by Attouch et al. \cite{HA-ZC-JF-HR:22b} and Bo\c{t} and Nguyen \cite{RIB-DKN:21} using the second-order non-autonomous differential system
	\begin{equation}
	\renewcommand{\theequation}{AAH}\tag{\theequation}\label{sec5:accarrowhurwicz}
		\begin{cases}
			\hspace{30.5pt}\ddot{x}+\dfrac{\nu}{t}\dot{x}+\nabla_{x}L_{\mu}
			(x,y,\lambda+\theta t\dot{\lambda})=0_{X}\\[1.5ex]
			\hspace{31.5pt}\ddot{y}+\dfrac{\nu}{t}\dot{y}+\nabla_{y}L_{\mu}
			(x,y,\lambda+\theta t\dot{\lambda})=0_{Y}\\[1.5ex]
			\ddot{\lambda}+\dfrac{\nu}{t}\dot{\lambda}-\nabla_{\lambda}L_{\mu}
			(x+\theta t\dot{x},y+\theta t\dot{y},\lambda)=0_{Z}
		\end{cases}
	\end{equation}
	with $\nu\geq3$, $\mu\geq0$, $\theta\in[1/(\nu-1),1/2]$, and initial data $(x_{0},y_{0},\lambda_{0}),(v_{0},w_{0},\mu_{0})\in X\times Y\times Z$. As a decisive feature, the \eqref{sec5:accarrowhurwicz} dynamics are governed by `asymptotically vanishing damping coefficients' which relate the above system to Nesterov's accelerated gradient method (see Nesterov \cite{YN:83}, Su et al. \cite{WS-SB-EC:16}), and additional `exploration terms' within the partial gradients of the augmented Lagrangian $L_{\mu}$ associated with \eqref{sec5:strucproblem}. This particular structure allows for remarkably fast mini-maximizing properties with respect to the Lagrangian $L$ given the sole convexity hypothesis on the objective function of \eqref{sec5:strucproblem}. In particular, the (classical) solutions $(x,y,\lambda):[t_{0},+\infty)\to X\times Y\times Z$ of \eqref{sec5:accarrowhurwicz} with $t_{0}>0$ evolve, for any $(\xi,\psi,\eta)\in S\times U\times M$, according to the asymptotic estimate (see Attouch et al. \cite{HA-ZC-JF-HR:22b}, Bo\c{t} and Nguyen \cite{RIB-DKN:21})
	\begin{equation*}
		L(x(t),y(t),\eta)-L(\xi,\psi,\lambda(t))=\OO\Big(\frac{1}{t^{2}}\Big)\
		\text{as $t\to+\infty$}.
	\end{equation*}
	If, in addition, $\nu>3$ and $\theta\in(1/(\nu-1),1/2]$, then the solutions of \eqref{sec5:accarrowhurwicz} further re- main bounded on $[t_{0},+\infty)$ and it holds that
	\begin{equation*}
		\norm{(\dot{x}(t),\dot{y}(t),\dot{\lambda}(t))}=\OO\Big(\frac{1}{t}\Big)\
		\text{as $t\to+\infty$}.
	\end{equation*}
	Assuming, moreover, that $f+g:X\times Y\to\rl$ is strongly convex, then, for any $(\xi,\psi,\eta)\in S\times U\times M$, the following estimate is verified (see Attouch et al. \cite{HA-ZC-JF-HR:22b}):
	\begin{equation*}
		\norm{(x(t),y(t))-(\xi,\psi)}_{X\times Y}=\OO\Big(\frac{1}{t}\Big)\ \text{as $t\to+\infty$}.
	\end{equation*}
	The latter may be particularized to an exponential estimate by further introducing tem- poral scaling factors in \eqref{sec5:accarrowhurwicz}; cf. Attouch et al. \cite[Remark 5.2]{HA-ZC-JF-HR:22b}.

	In view of the above discussion, we observe that the second-order differential system \eqref{sec5:accarrowhurwicz} clearly outperforms the first-order differential system \eqref{sec1:arrowhurwicz} in the case of a con- vex objective function. However, this may not be the case, as we shall see next, whenever the objective function is strongly convex.
\end{remark}

\section{Numerical experiments}\label{sec6}
In this section, we perform numerical experiments on the Arrow--Hurwicz differential system \eqref{sec1:arrowhurwicz} to support our theoretical findings. In particular, we consider two simple but representative (strongly) convex minimization problems in two dimensions.

\begin{example}
	Let $X,Y=\rl^{2}$ and consider the quadratic function $f:\rl^{2}\to\rl$ defined by $f(x_{1},x_{2})=(x_{1}^{2}-x_{1}x_{2}+x_{2}^{2})/2$. Clearly, $f$ is $\alpha$-strongly convex with $\alpha=1/2$. Moreover, $\nabla^{2}f(\,\cdot\,)$ is $\gamma$-bounded with $\gamma=3/2$. Further, let $A(x_{1},x_{2})=(x_{1},x_{2})$, $b=(1,1)$, and observe that $A$ is bounded from below with constant $\beta=1$. The unique minimizer of $f$ subject to the linear constraints corresponds to $\xi=(1,1)$; with associated Lagrange multiplier $\eta=(-1/2,-1/2)$. The evolution of the primal-dual gap function $L(x(t),\eta)-L(\xi,\lambda(t))$, the squared velocity $\norm{(\dot{x}(t),\dot{\lambda}(t))}^{2}$, the squared error $\normX{x(t)-\xi}^{2}$, and the trajectory of the solution component $x(t)=(x_{1}(t),x_{2}(t))$ of \eqref{sec1:arrowhurwicz} with initial data $x_{0}=(-1,1)$ and $\lambda_{0}=(1,1)$ is depicted in Figure \refeq{sec6:fig1:scvxprog}. For comparison, the corresponding quantities of the \eqref{sec5:accarrowhurwicz} differential system are displayed with damping parameter $\nu=3$, exploration coefficient $\theta=1/2$, and augmentation parameter $\mu=1/2$. The initial data of the \eqref{sec5:accarrowhurwicz} differential system is set accordingly to $x_{0}=(-1,1)$, $\lambda_{0}=(1,1)$, $v_{0}=(0,0)$, and $\mu_{0}=(1,1)$.
	\begin{figure}[t]
		\centering
		\subfloat{\includegraphics[trim=12pt 0pt 52pt 23pt,clip,
			width=.49\linewidth]{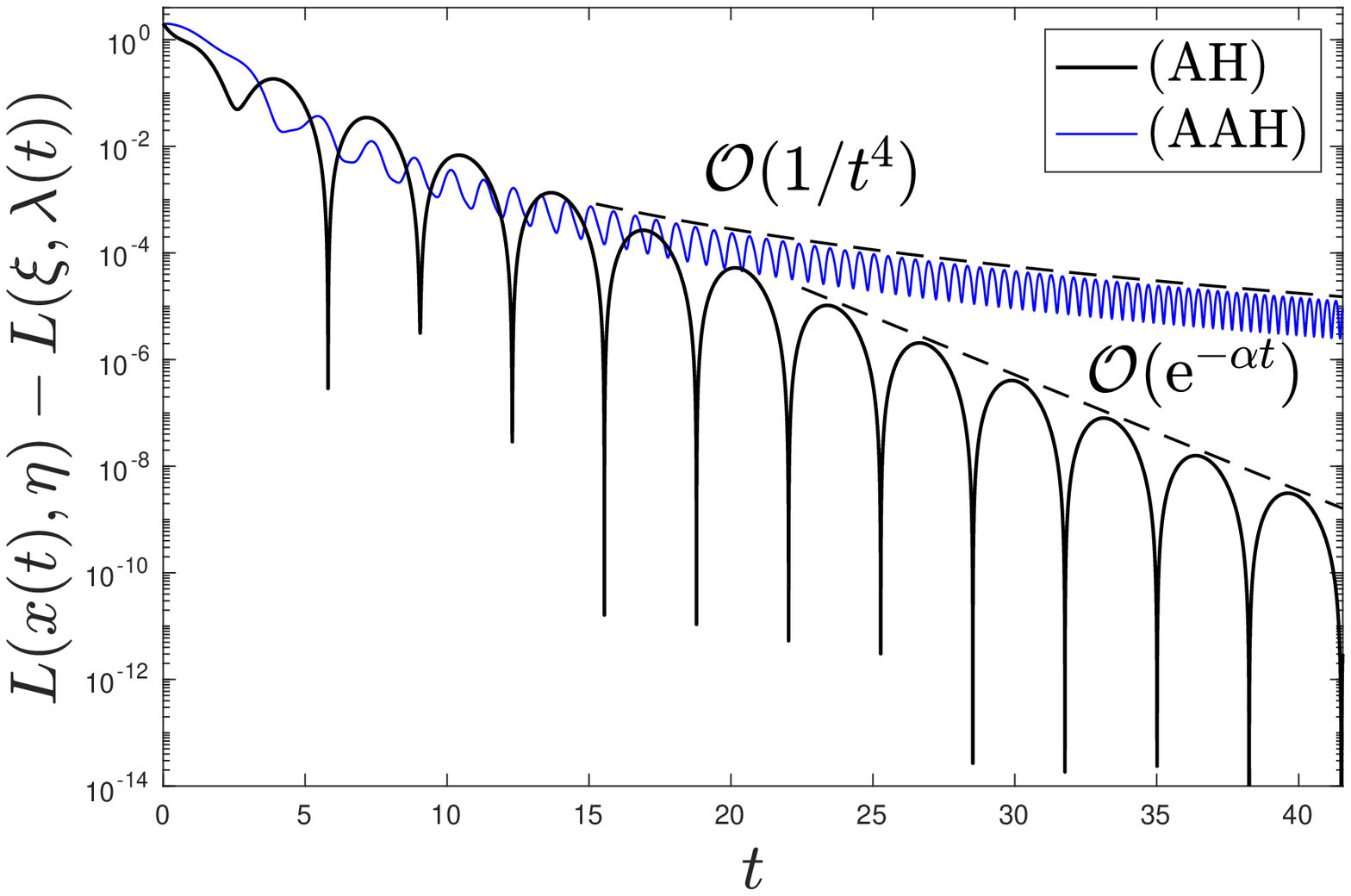}}%
		\hfill
		\subfloat{\includegraphics[trim=12pt 0pt 52pt 23pt,clip,
			width=.49\linewidth]{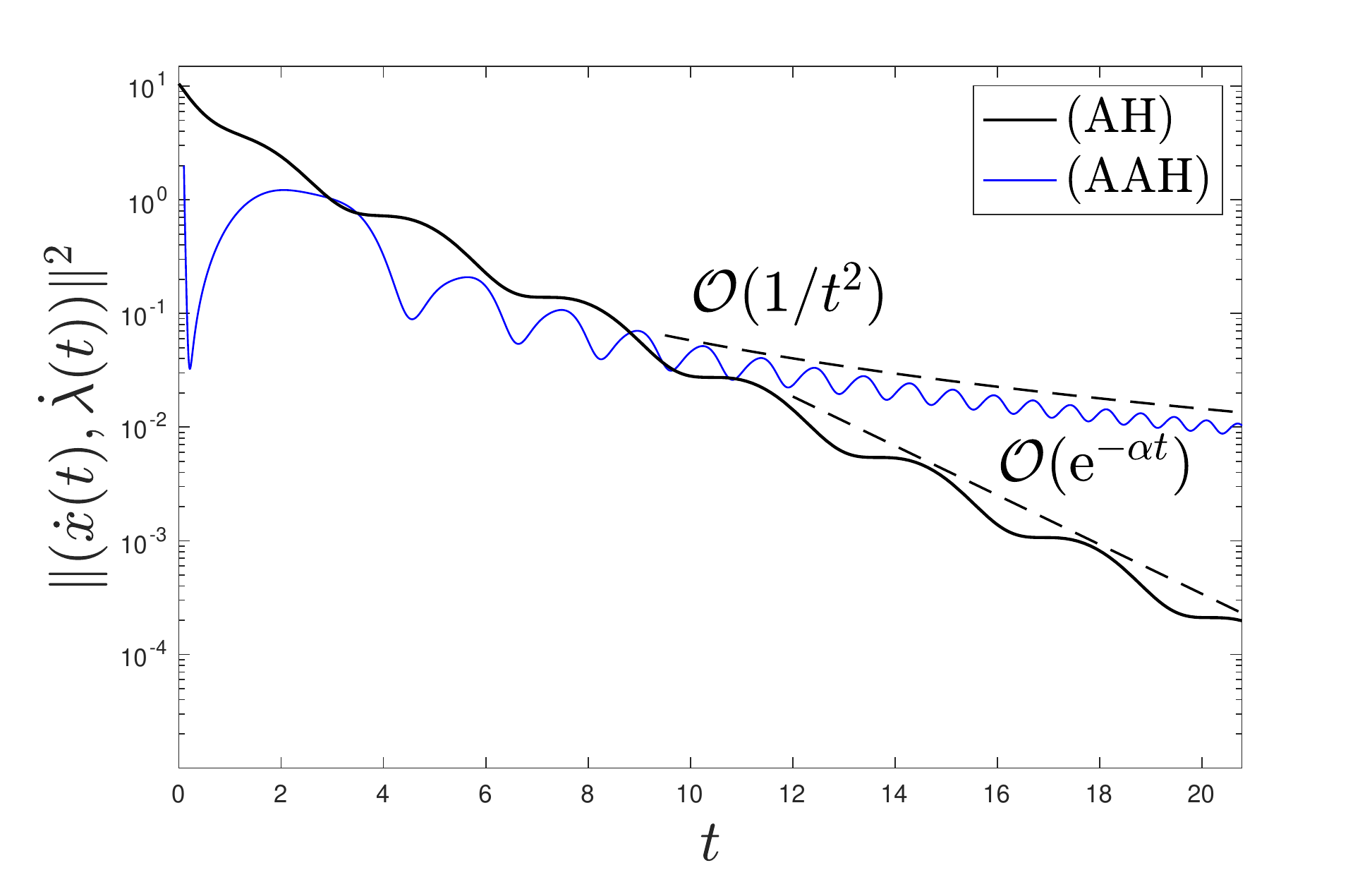}}%
		\hfill
		\subfloat{\includegraphics[trim=12pt 0pt 52pt 23pt,clip,
			width=.49\linewidth]{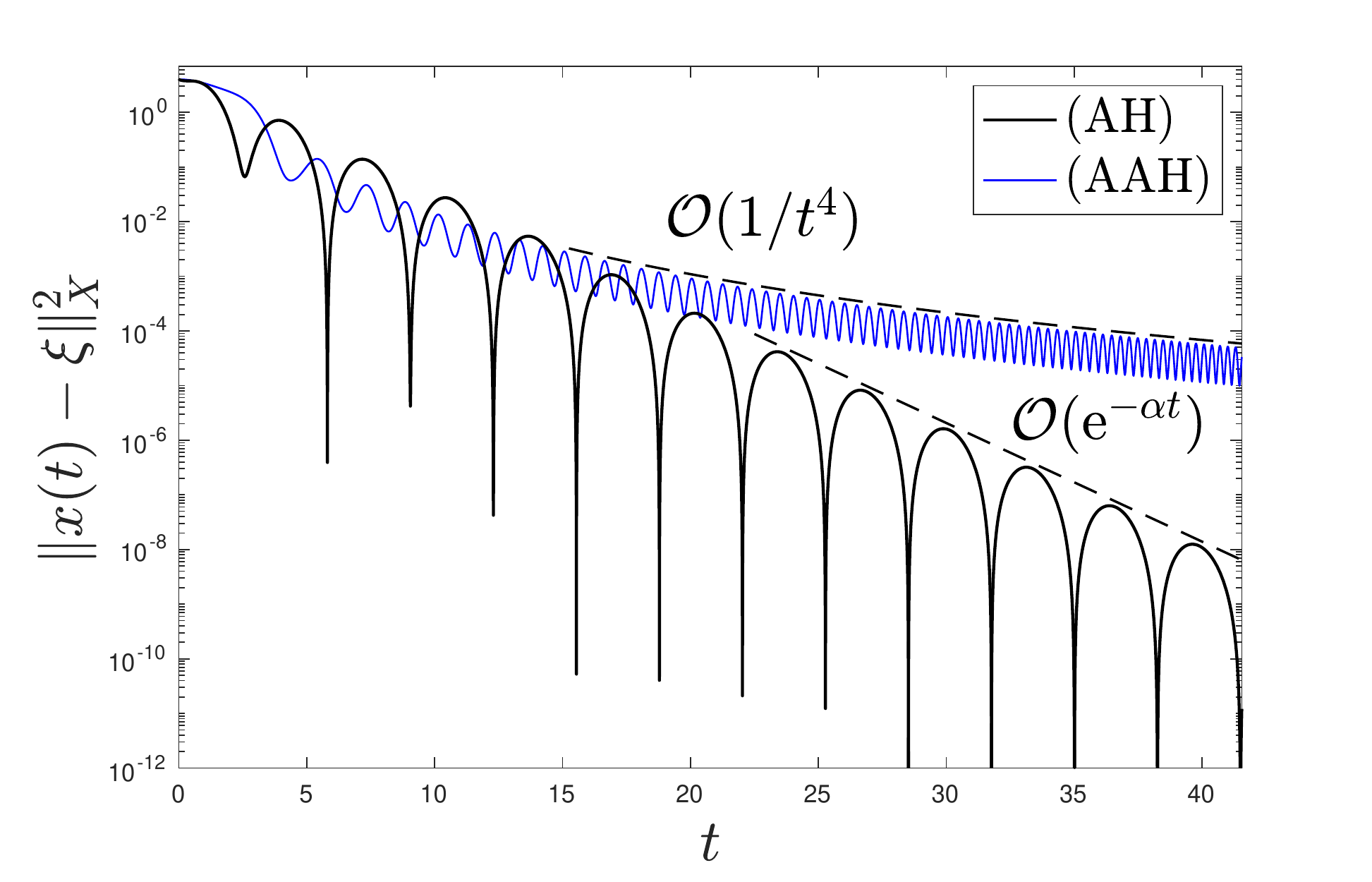}}%
		\hfill
		\subfloat{\includegraphics[trim=12pt 0pt 52pt 23pt,clip,
			width=.49\linewidth]{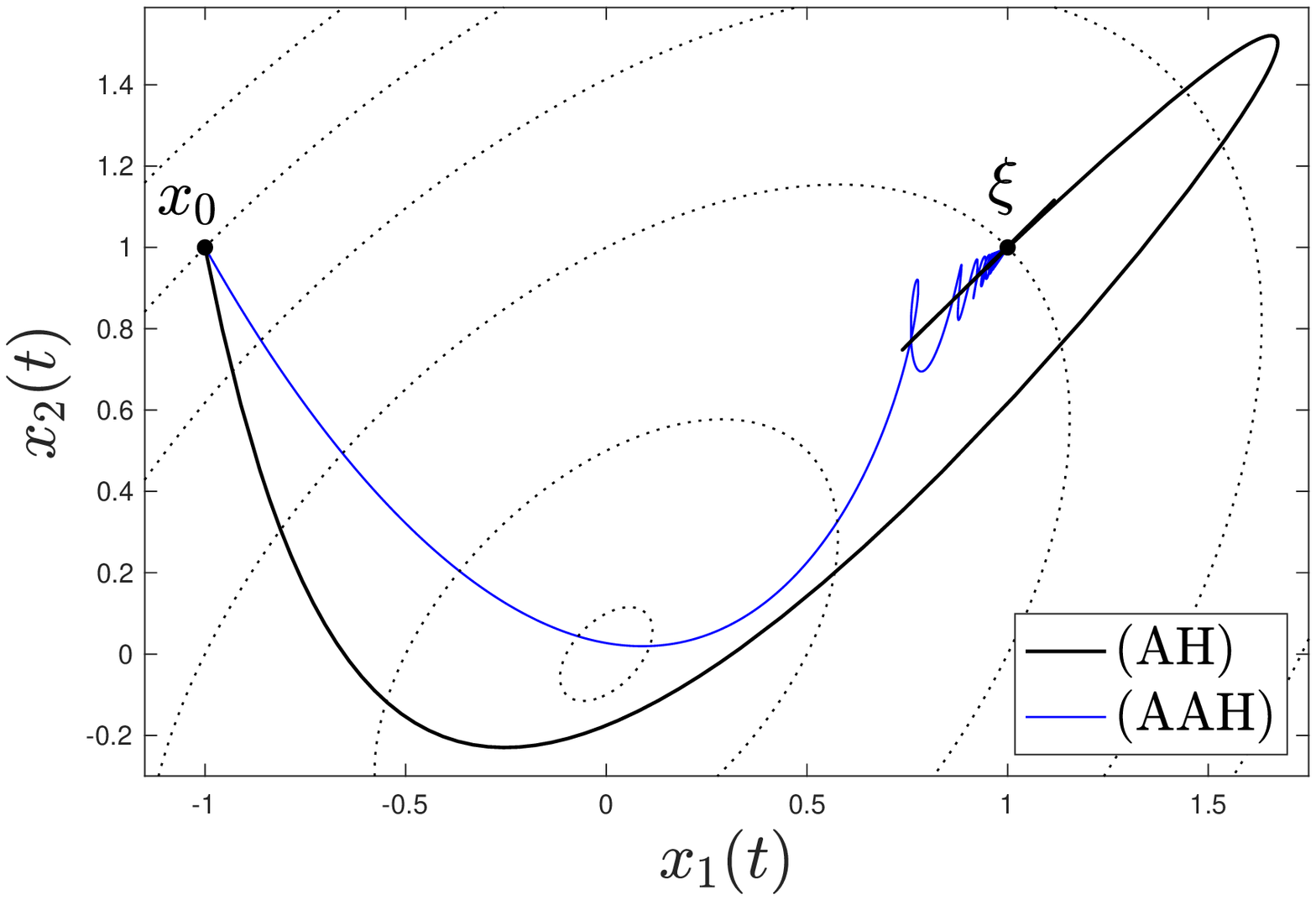}}%
		\vspace{-2pt}
		\caption{%
			Graphical view on the evolution of the primal-dual gap function $L(x(t),\eta)-L(\xi,\lambda(t))$, the squared velocity $\norm{(\dot{x}(t),\dot{\lambda}(t))}^{2}$, the squared error $\normX{x(t)-\xi}^{2}$, and the trajectories of the solution components $x(t)=(x_{1}(t),x_{2}(t))$ of \eqref{sec1:arrowhurwicz} and \eqref{sec5:accarrowhurwicz}.
		}
		\label{sec6:fig1:scvxprog}
	\end{figure}

	Analyzing Figure \refeq{sec6:fig1:scvxprog}, we observe that the solutions $(x(t),\lambda(t))$ of \eqref{sec1:arrowhurwicz} converge, as $t\to+\infty$, towards the unique mini-maximizer $(\xi,\eta)$ of the convex minimization problem \eqref{sec1:cvxproblem} and its associated Lagrange dual \eqref{sec1:dualcvxproblem}; cf. Proposition \ref{sec3:pr:strongconv}. Moreover, according to Theorem \ref{sec4:th:expconv}{\it (i)}, we find that the primal-dual gap function $L(x(t),\eta)-L(\xi,\lambda(t))$, the squared velocity $\norm{(\dot{x}(t),\dot{\lambda}(t))}^{2}$ and the squared error $\normX{x(t)-\xi}^{2}$ obey the exponential estimate $\OO\big(\mathrm{e}^{-\alpha t}\big)$ as $t\to+\infty$. Compared to the \eqref{sec5:accarrowhurwicz} dynamics for which the quantity $\norm{(\dot{x}(t),\dot{\lambda}(t))}^{2}$ evolves according to the estimate $\OO\big(1/t^{2}\big)$ as $t\to+\infty$ (even though the damping parameter is chosen to be $\nu=3$), we find that the solutions of \eqref{sec1:arrowhurwicz} indeed admit a faster and less oscillatory decay. It is interesting to note that, in this example, the primal-dual gap function $L(x(t),\eta)-L(\xi,\lambda(t))$ and the squared error $\normX{x(t)-\xi}^{2}$ for \eqref{sec5:accarrowhurwicz} appear to obey the estimate $\OO\big(1/t^{4}\big)$ rather than $\OO\big(1/t^{2}\big)$ as $t\to+\infty$.
\end{example}

\begin{example}
	Let $X,Y=\rl^{2}$ and consider the parameterized quadratic function $f:\rl^{2}\to\rl$ defined by $f(x_{1},x_{2})=\alpha(x_{1}^{2}+x_{2}^{2})/2$ with $\alpha>0$. Further, let $A(x_{1},x_{2})=\sqrt{2}(x_{1}+x_{2})/2$ and $\beta=\sqrt{2}/2$ so that $A^{\ast}$ is bounded from below with constant $\beta=1$. The unique minimizer of $f$ subject to the linear constraints is denoted by $\xi$; with cor- responding Lagrange multiplier $\eta$. Figure \refeq{sec6:fig2:decayrate} illustrates the decay properties of the primal-dual gap function $L(x(t),\eta)-L(\xi,\lambda(t))$, the squared velocity $\norm{(\dot{x}(t),\dot{\lambda}(t))}^{2}$, and the squared error $\norm{(x(t),\lambda(t))-(\xi,\eta)}^{2}$ of the solutions $(x(t),\lambda(t))$ of \eqref{sec1:arrowhurwicz} for the distinct values $\alpha=1$, $\alpha=2$, and $\alpha=3$. The initial data is set to $x_{0}=(-1,1)$ and $\lambda_{0}=1$.
	\begin{figure}[t]
		\centering
		\subfloat{\includegraphics[trim=12pt 0pt 52pt 23pt,clip,
			width=.49\linewidth]{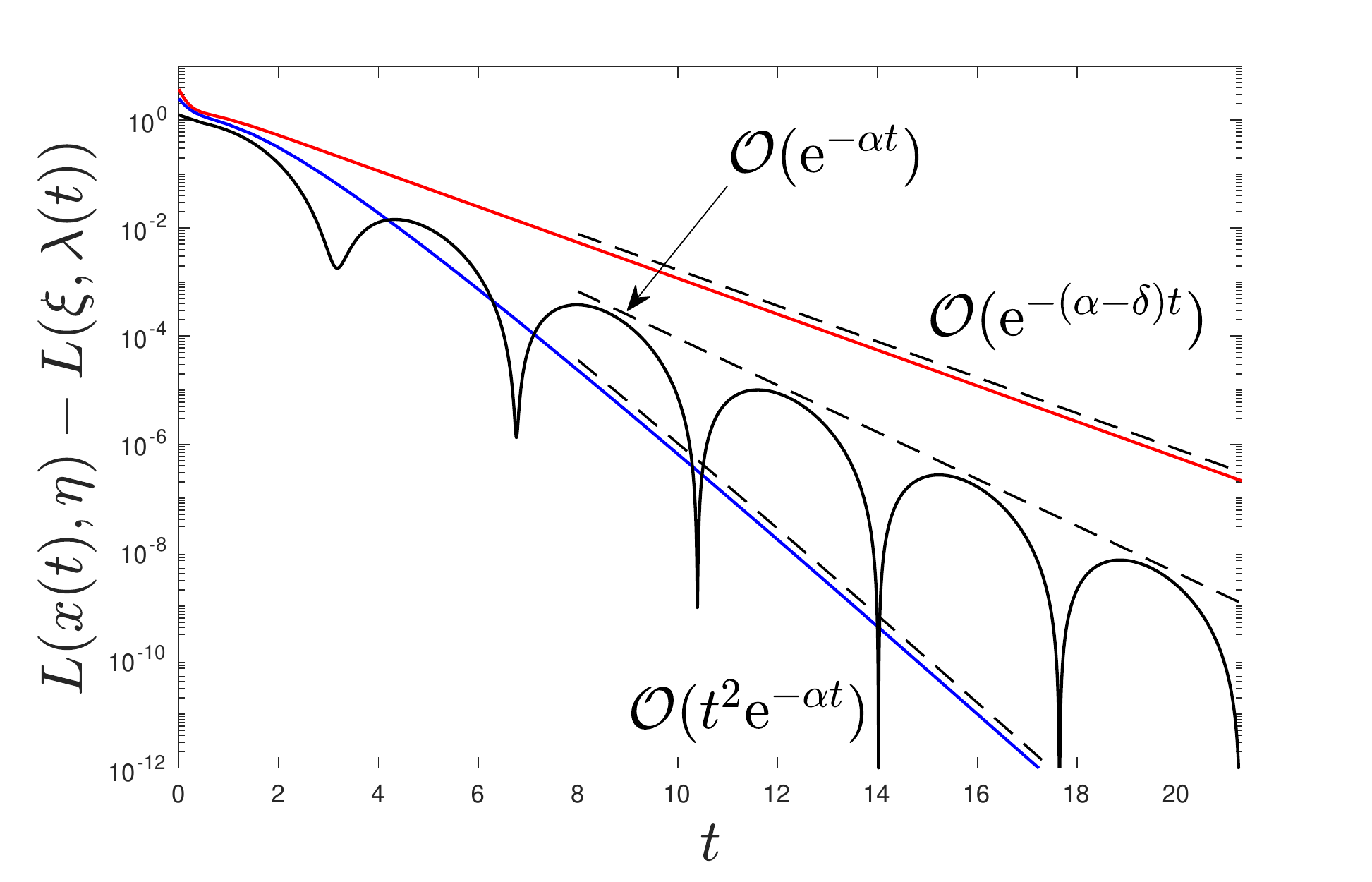}}%
		\hfill
		\subfloat{\includegraphics[trim=12pt 0pt 52pt 23pt,clip,
			width=.49\linewidth]{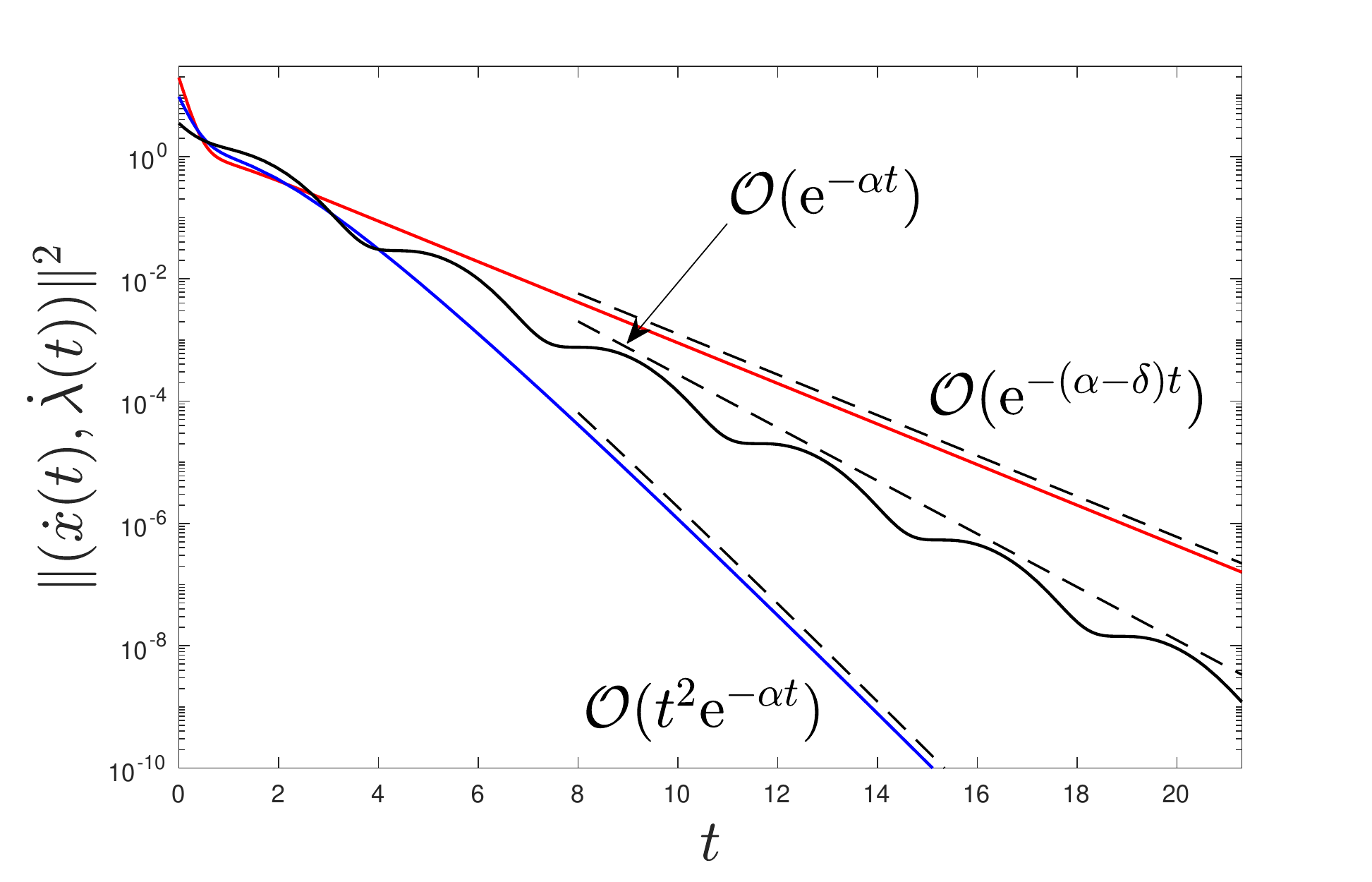}}%
		\hfill
		\subfloat{\includegraphics[trim=12pt 0pt 52pt 23pt,clip,
			width=.49\linewidth]{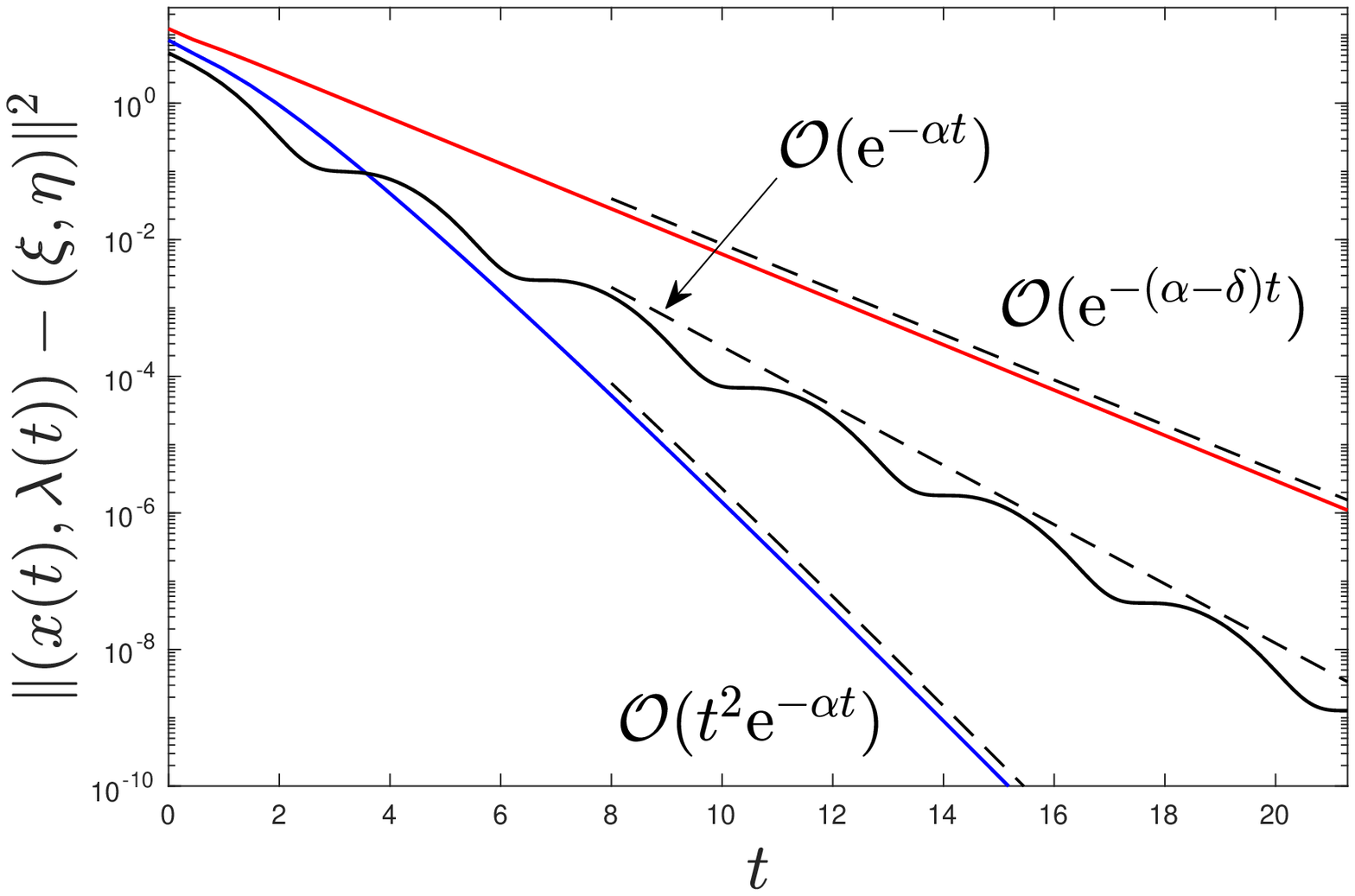}}%
		\hfill
		\subfloat{\includegraphics[trim=12pt 0pt 52pt 23pt,clip,
			width=.49\linewidth]{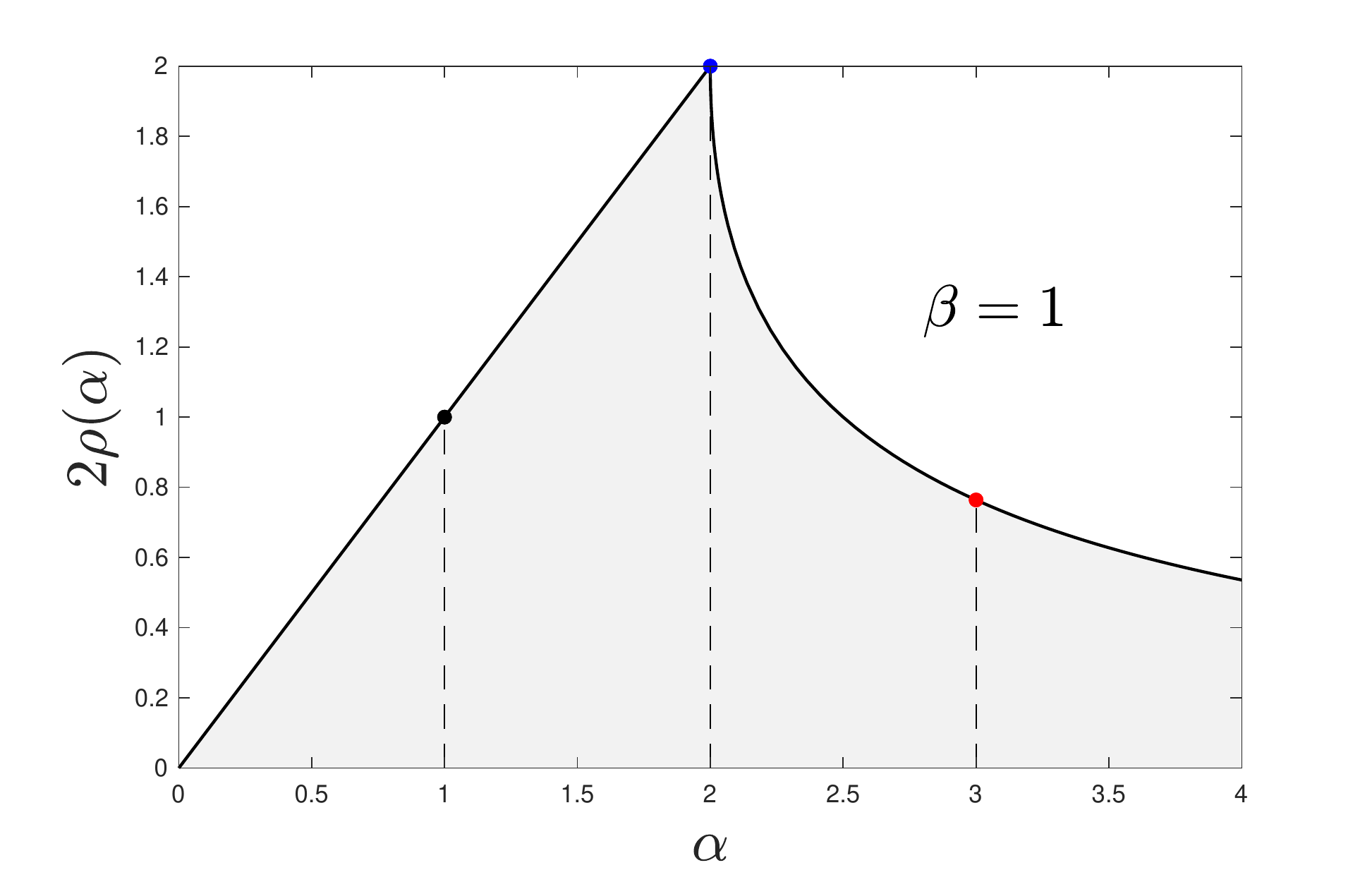}}%
		\vspace{-2pt}
		\caption{%
			Exponential decay properties of the primal-dual gap function $L(x(t),\eta)-L(\xi,\lambda(t))$, the squared velocity $\norm{(\dot{x}(t),\dot{\lambda}(t))}^{2}$, and the squared error $\norm{(x(t),\lambda(t))-(\xi,\eta)}^{2}$ of the solutions $(x(t),\lambda(t))$ of \eqref{sec1:arrowhurwicz} for distinct values of $\alpha$.
		}
		\label{sec6:fig2:decayrate}
	\end{figure}

	Figure \refeq{sec6:fig2:decayrate} suggests that the solutions $(x(t),\lambda(t))$ of \eqref{sec1:arrowhurwicz} converge, as $t\to+\infty$, at an exponential rate towards the unique mini-maximizer $(\xi,\eta)$ of the convex minimization problem \eqref{sec1:cvxproblem} and its associated Lagrange dual \eqref{sec1:dualcvxproblem}. Indeed, the decay properties of the solutions of \eqref{sec1:arrowhurwicz} may be categorized as predicted by Corollary \refeq{sec4:co:dualexpconv}: In case {\it (i)}, we have $\alpha^{2}<4\beta$ with the rate estimate $\OO\big(\mathrm{e}^{-\alpha t}\big)$ as $t\to+\infty$. We refer to this case as the `under-damped case' as the solutions of \eqref{sec1:arrowhurwicz} admit a significant oscillatory behavior. In case {\it (ii)}, we have $\alpha^{2}=4\beta$ with the rate estimate $\OO\big(t^{2}\mathrm{e}^{-\alpha t}\big)$ as $t\to+\infty$. This case refers to the `critically-damped case' for which we observe the fastest possible convergence of the solutions of \eqref{sec1:arrowhurwicz}. Finally, in case {\it (iii)}, we have $\alpha^{2}>4\beta$ with the rate estimate $\OO\big(\mathrm{e}^{-(\alpha-\delta)t}\big)$ as $t\to+\infty$, where $\delta=\sqrt{\alpha^{2}-4\beta}$. In this case, referred to as the `over-damped case', the decay of the solutions of \eqref{sec1:arrowhurwicz} is considerably degraded.
\end{example}

\section*{Acknowledgment}
The author expresses his gratitude to the two anonymous reviewers whose comments and suggestions led to a significant improvement of this manuscript.

\section*{Disclosure statement}
No potential conflict of interest was reported by the author.

\section*{Funding}
Research supported by the German Research Foundation (DFG).

\bibliographystyle{tfnlm}
\bibliography{../../../bib/alias,../../../bib/sn}

\begin{thebibliography}{10}
\providecommand{\url}[1]{\normalfont{#1}}
\providecommand{\urlprefix}{Available from: }

\bibitem{IE-RT:99}
Ekeland~I, T{\'e}mam~R. Convex analysis and variational problems. Philadelphia:
  Society for Industrial and Applied Mathematics; 1999. Classics in applied
  mathematics.

\bibitem{JBHU-CL:93}
Hiriart-Urruty~JB, Lemar{\'e}chal~C. Convex analysis and minimization
  algorithms {I}. New York: Springer; 1993. Grundlehren der mathematischen
  Wissenschaften 305.

\bibitem{HHB-PLC:17}
Bauschke~HH, Combettes~PL. Convex analysis and monotone operator theory in
  {H}ilbert spaces. New York: Springer; 2017. CMS Books in Mathematics.

\bibitem{KJA-LH:51}
Arrow~KJ, Hurwicz~L. A gradient method for approximating saddle points and
  constrained maxima. RAND Corp, Santa Monica, CA. 1951;P--223.

\bibitem{TK:56}
Kose~T. Solutions of saddle value problems by differential equations.
  Econometrica. 1956; 24:59--70.

\bibitem{KJA-LH-HU:58}
Arrow~KJ, Hurwicz~L, Uzawa~H. Studies in linear and non-linear programming.
  Stanford, CA: Stanford University Press; 1958.

\bibitem{GJM:62}
Minty~GJ. Monotone (nonlinear) operators in {H}ilbert space. Duke Math J.
  1962;\hspace{0pt}29:341--346.

\bibitem{RTR:69a}
Rockafellar~RT. Monotone operators associated with saddle-functions and minimax
  problems. In Nonlinear Functional Analysis, {\it Proceedings of Symposia in
  Pure Math}, Amer Math Soc. 1969;241--250.

\bibitem{RTR:71}
Rockafellar~RT. Saddle-points and convex analysis. In Differential Games and
  Related Topics, North-Holland. 1971;109--127.

\bibitem{MGC-AP:69}
Crandall~MG, Pazy~A. Semi-groups of nonlinear contractions and dissipative
  sets. J Funct Anal. 1969;3:376--418.

\bibitem{HB:73}
Br{\'e}zis~H. Op{\'e}rateurs maximaux monotones et semi-groupes de contractions
  dans les espaces de {H}ilbert. Amsterdam: North-Holland; 1973. Mathematics
  Studies 5.

\bibitem{TK:67}
Kato~T. Nonlinear semigroups and evolution equations. J Math Soc Japan.
  1967;19:508--520.

\bibitem{YK:67}
K{\=o}mura~Y. Nonlinear semi-groups in {H}ilbert space. J Math Soc Japan.
  1967;19:493--507.

\bibitem{FEB:76}
Browder~FE. Nonlinear operators and nonlinear equations of evolution in
  {B}anach spaces. In Nonlinear Functional Analysis, {\it Proceedings of
  Symposia in Pure Math}, Amer Math Soc. 1976.

\bibitem{JBB-HB:76}
Baillon~JB, Br{\'e}zis~H. Une remarque sur le comportement asymptotique des
  semigroupes non lin{\'e}aires. Houston J Math. 1976;2:5--7.

\bibitem{VIV:85}
Venets~VI. Continuous algorithms for solution of convex optimization problems
  and finding saddle points of convex-concave functions with the use of
  projection operators. Optimization. 1985;16:519--533.

\bibitem{SDF-ABI:89}
Fl\aa{}m~SD, Ben-Israel~A. Approximating saddle points as equilibria of
  differential inclusions. J Math Anal Appl. 1989;141:264--277.

\bibitem{JPA-AC:84}
Aubin~JP, Cellina~A. Differential inclusions. New York: Springer; 1984.
  Grundlehren der mathematischen Wissenschaften 264.

\bibitem{BTP:70}
Polyak~BT. Iterative methods using {L}agrange multipliers for solving extremal
  problems with constraints of the equation type. USSR Comput Math and Math
  Phys. 1970;10:42--52.

\bibitem{ASN-DBY:78}
Nemirovski~AS, Yudin~DB. Cesari convergence of the gradient method of
  approximating saddle points of convex-concave functions. Dokl Akad Nauk SSSR.
  1978;239:1056--1059.

\bibitem{LMB:67}
Bregman~LM. The relaxation method of finding the common point of convex sets
  and its application to the solution of problems in convex programming. USSR
  Comput Math and Math Phys. 1967;7:200--217.

\bibitem{AH:91}
Haraux~A. Syst{\`e}mes dynamiques dissipatifs et applications. Masson, Paris:
  Recherches en Math{\'e}matiques Appliqu{\'e}es 17; 1991.

\bibitem{AP:79}
Pazy~A. Semi-groups of nonlinear contractions and their asymptotic behavior. In
  Nonlinear Analysis and Mechanics: Heriot-Watt Symposium III, Pitman, London.
  1979;36--134.

\bibitem{JP-SS:10}
Peypouquet~J, Sorin~S. Evolution equations for maximal monotone operators:
  {A}symptotic analysis in continuous and discrete time. J Convex Anal.
  2010;17:1113--1163.

\bibitem{GBP:79}
Passty~GB. Ergodic convergence to a zero of the sum of monotone operators in
  {H}ilbert space. J Math Anal Appl. 1979;72:383--390.

\bibitem{HB:78}
Br{\'e}zis~H. Asymptotic behavior of some evolution systems. In Nonlinear
  Evolution Equations, Madison, 1977, Acad Press. 1978;141--154.

\bibitem{ZO:67}
Opial~Z. Weak convergence of the sequence of successive approximations for
  nonexpansive mappings. Bull Amer Math Soc. 1967;73:591--597.

\bibitem{REB:75}
Bruck~RE. Asymptotic convergence of nonlinear contraction semi-groups in
  {H}ilbert spaces. J Funct Anal. 1975;18:15--26.

\bibitem{ZC-HR:14}
Chbani~Z, Riahi~H. Existence and asymptotic behaviour for solutions of
  dynamical equilibrium systems. Evol Equ Control Theory.
  2014;3:1--14.

\bibitem{HB:11}
Br{\'e}zis~H. Functional analysis, {S}obolev spaces and partial differential
  equations. New York: Springer; 2011.

\bibitem{MB-GHG-JL:05}
Benzi~M, Golub~GH, Liesen~J. Numerical solution of saddle point problems. Acta
  Numer. 2005;14:1--137.

\bibitem{BTP:64}
Polyak~BT. Some methods of speeding up the convergence of iteration methods.
  USSR Comput Math and Math Phys. 1964;4:1--17.

\bibitem{FA:00}
{\'A}lvarez~F. On the minimizing property of a second order dissipative system
  in {H}ilbert spaces. SIAM J Control Optim. 2000;38:1102--1119.

\bibitem{HA-XG-PR:00}
Attouch~H, Goudou~X, Redont~P. The heavy ball with friction method, {I}. {T}he
  continuous dynamical system: {G}lobal exploration of the local minima of a
  real-valued function by asymptotic analysis of a dissipative dynamical
  system. Commun Contemp Math. 2000; 2:1--34.

\bibitem{HA-ZC-JF-HR:22b}
Attouch~H, Chbani~Z, Fadili~J, et~al. Fast convergence of dynamical {ADMM} via
  time scaling of damped inertial dynamics. J Optim Theory Appl.
  2022;193:704--736.

\bibitem{RIB-DKN:21}
Bo{\c{t}}~RI, Nguyen~DK. Improved convergence rates and trajectory convergence
  for primal-dual dynamical systems with vanishing damping. J Differ Equ.
  2021;303:369--406.

\bibitem{YN:83}
Nesterov~Y. A method of solving a convex programming problem with convergence
  rate $\mathcal{O}(1/k^{2})$. Sov Math Dokl. 1983;27:372--376.

\bibitem{WS-SB-EC:16}
Su~W, Boyd~S, Cand{\`e}s~E. A differential equation for modeling {N}esterov's
  accelerated gradient method: {T}heory and insights. J Mach Learn Res.
  2016;17:1--43.

\end{thebibliography}

\end{document}